\pdfoutput=1
\documentclass[a4paper]{article}

\usepackage{fullpage}
\usepackage{microtype}

\usepackage[sorting=nyt]{biblatex}
\usepackage{graphicx}
\graphicspath{ {./images/} }

\usepackage{enumitem}

\usepackage{booktabs}
\usepackage{longtable}

\usepackage{amsthm}
\newtheorem*{theorem*}{Theorem}
\usepackage{amsmath}
\DeclareMathOperator{\lcm}{lcm}

\usepackage[dvipsnames]{xcolor}
\definecolor{darkgreen}{rgb}{0,0.45,0}
\usepackage[colorlinks,citecolor=darkgreen,linkcolor=NavyBlue,urlcolor=NavyBlue]{hyperref}

\newcommand{\patcols}[3]{
\noindent\begin{minipage}[t]{0.2\textwidth}
    \vspace{0pt}
    #2
  \end{minipage}\hfill
  \begin{minipage}[t]{0.75\textwidth}
    \vspace{0pt}
    \raggedright
    {\noindent\normalfont\bfseries #1} \\
    \vspace{3pt}
    \small
    \texttt{#3}
  \end{minipage}
  \vspace{50pt}
}

\newcommand{\patstack}[3][0.15]{
\begin{minipage}[t]{#1\textwidth}
    \begin{center}
    \vspace{0pt}
    #3 \\
    {\noindent\normalfont\bfseries #2}
    \end{center}
  \end{minipage}
}

\usepackage[status=draft,inline,nomargin]{fixme}
\FXRegisterAuthor{mvr}{anmvr}{\color{blue}MVR}
\FXRegisterAuthor{mk}{anmk}{\color{green}Maia}
\FXRegisterAuthor{sok}{ansok}{\color{violet}Sokwe}
\FXRegisterAuthor{cc}{acc}{\color{red}carsoncheng}
\FXRegisterAuthor{pzq}{anpzq}{\color{cyan}pzq}

\title{Conway's Game of Life is Omniperiodic}
\date{December 5, 2023}

\author{
Nico Brown \vspace{-0.3em} \\ {\small \href{mailto:nicobrownmath@gmail.com}{\texttt{nicobrownmath@gmail.com}}} \vspace{0.5em} \\
Maia Karpovich \vspace{-0.3em} \\ {\small \href{mailto:hcivoprak@gmail.com}{\texttt{hcivoprak@gmail.com}}} \and 
Carson Cheng \vspace{-0.3em} \\ {\small \href{mailto:carsonwhatif@gmail.com}{\texttt{carsonwhatif@gmail.com}}}\vspace{0.5em} \\
Matthias Merzenich \vspace{-0.3em} \\ {\small \href{mailto:merzenich.matthias@gmail.com}{\texttt{merzenich.matthias@gmail.com}}}\vspace{0.5em} \\
Mitchell Riley\thanks{Contact author.} \vspace{-0.3em} \\ {\small \href{mailto:mitchell.v.riley@gmail.com}{\texttt{mitchell.v.riley@gmail.com}}} \and
Tanner Jacobi \vspace{-0.3em} \\ {\small \href{mailto:tanner.g.jacobi@gmail.com}{\texttt{tanner.g.jacobi@gmail.com}}}\vspace{0.5em} \\
David Raucci \vspace{-0.3em} \\ {\small \href{mailto:davidraucci122@gmail.com}{\texttt{davidraucci122@gmail.com}}}
}

\begin{document}

\maketitle

In the theory of cellular automata, an \emph{oscillator} is a pattern that repeats itself after a fixed number of generations; that number is called its \emph{period}. A cellular automaton is \emph{omniperiodic} if there exist oscillators of all periods.

Conway's Game of Life is by far the most famous cellular automaton. David Buckingham first established a finite bound above which oscillators of every period could be built by running a signal around a specially constructed track. For periods below that cutoff, oscillators need to be found individually.

At the turn of the millennium, only twelve periods remained to be found: 19, 23, 27, 31, 34, 37, 38, 39, 41, 43, 51 and 53. Over the last couple of decades, these periods were gradually filled in as increasing computer speed and more clever search techniques were brought to bear on the problem.

\medskip
\noindent The search has finally ended, with the discovery of oscillators having the final two periods, 19 and 41. A table with oscillators of all periods is provided \hyperref[sec:table]{below}, proving:

\begin{theorem*}
  Life is omniperiodic. \qed
\end{theorem*}

\section{Conway's Game of Life}

Conway's Game of Life~\cite{gardner:1970} is a cellular automaton occurring on an infinite plane of square grid cells, each of which is in one of two states: alive or dead. The neighbourhood of a cell is the 8 cells that are connected orthogonally or diagonally to it. At each time step, or generation, the entire plane changes state according to the following rules:

\begin{itemize}[noitemsep,topsep=3pt]
\item If a dead cell has exactly three live neighbours, it becomes alive.
\item If an alive cell has exactly two or three live neighbours, it stays alive. Otherwise, it becomes dead.
\end{itemize}

These simple rules give rise to complex behavior which has been studied in great detail over the last 50 years. Starting with a random initial state, running Life forwards in time typically causes an initial burst of chaotic activity that settles down into small patterns of alive cells. Two basic arrangements are the block, a 2\texttimes{}2 grid of alive cells, which does not change; and the blinker, a 1\texttimes{}3 line of alive cells, which alternates between horizontal to vertical at each time step. Vastly more complex patterns exist, both naturally occurring and intentionally engineered. In particular, this includes working universal constructors~\cite[\S 11]{life:book} and computers~\cite[\S 9]{life:book} \cite{rendell:turing-machine}.

\begin{center}
\hfill
\patstack{(p1) Block}{\href{https://conwaylife.com/?rle=2o$2o!&name=(p1) block}{\includegraphics[width=\textwidth]{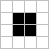}}}
\hfill
\patstack{(p2) Blinker}{\href{https://conwaylife.com/?rle=3o!&name=(p2) blinker}{\includegraphics[width=\textwidth]{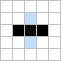}}}
\hfill
\patstack{Glider}{\href{https://conwaylife.com/?rle=bo$2bo$3o!&name=glider}{\includegraphics[width=\textwidth]{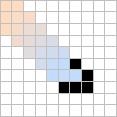}}}
\hfill
~
\end{center}

Some Life patterns return to their starting configuration after a finite amount of time. Such patterns are often categorised by their \emph{period}: the number of generations after which the state of the pattern repeats. Patterns that have a period of 1, like the block, are called \emph{still lifes}. Some patterns, like the blinker, have period 2 or more; these are called \emph{oscillators}. Others repeat their state but translate in the plane; such patterns are called \emph{spaceships}.  For example, the famous glider has period 4, but moves 1 cell diagonally in those 4 generations.

The final states of random soups often contain oscillators of period 2 (such as the blinker), sometimes period 3 (such as the pulsar), and occasionally even 15 (the pentadecathlon). Oscillators with other periods are much rarer, and are generally searched for specifically rather than falling out of random initial states. Nevertheless, it has long been conjectured that oscillators of \emph{every} period $p \geq 1$ exist in Life\footnote{The period of an oscillator is almost universally denoted by $p$ but we do not mean that $p$ is always prime, or that it is somehow sufficient to consider only prime periods.}; i.e. that Life is \emph{omniperiodic}. Determining whether this is the case has been referred to as an open problem as far back as we have records of hobbyist discussions~\cite{wechsler:missions}. 

One aspect that makes the problem so intriguing is that omniperodicity has long been settled for rules extremely similar to Life. For example, if the rules are changed so that an alive cell also remains alive when it has zero or one neighbours, omniperiodicity was proven in 1997~\cite{b3s0123-omniperiodic}. If instead we also allow four-neighbour cells to survive, the question was settled in 2010~\cite{b3s234-omniperiodic}. Both were achieved using ``drifter'' techniques (which we will be covering later). 

Low-period oscillators in Life, roughly $p \leq 15$, can be found by playing with patterns by hand or using brute force computer searches. In 1996, David Buckingham demonstrated~\cite{buckingham:conduits} using his ``Herschel conduits'' that one can construct oscillators with $p \geq 61$ by sending signals around a closed track; the cutoff for systematically constructing oscillators was later improved to $p \geq 43$, by Mike Playle's discovery of the Snark~\cite{snark}. 

Periods in the ``missing middle'', $15 < p < 43$, particularly those that are prime, proved more difficult to find. From 1996 and concluding in July 2023, a sequence of discoveries using new mechanisms and improved search algorithms filled in the remaining gaps. The techniques include: stabilizing infinitely repeating patterns, hassling common active objects, finding oscillating perturbations of stable background patterns, multiplying the period of low-period oscillators and finding dependent glider loops. Together, they suffice to cover all $p < 43$, proving that Life is omniperiodic.

\paragraph{Acknowledgements} Thank you to Dave Greene, Nathaniel Johnston, Darren Li and Conwaylife.com forum user `praosylen' for corrections and suggestions, and to Chris Rowett for the use of his LifeViewer software to depict each pattern. Thank you also to everyone who has contributed to the omniperiodicity problem over the years by making new discoveries, proposing and implementing new approaches, and running the searches that make these discoveries possible. The discoverers of the first oscillator of each period are listed in the \hyperref[sec:table]{Appendix}.

\section{A History of Oscillator Searches}

In the following sections we describe the techniques listed above, in roughly the chronological order in which they were developed. In the PDF version of this paper, each image links to a page with the pattern loaded into \href{https://lazyslug.com/lifeviewer/}{LifeViewer}, which can be used to simulate the pattern forward in time. There is no substitute for seeing the patterns in motion, and we encourage the reader to experiment with any pattern that interests them.

Oscillators are depicted with a colour ramp on the dead cells, fading from blue to orange the longer the cell has stayed dead. The number of generations taken to reach orange in each image is matched to the period of the oscillator, so the colour ramp can be used to roughly gauge where the active part of each oscillator is moving in the generation it is displayed: away from the blue region (where it has just been) and towards the orange region (which it must at some point reach).

\begin{center}
\patstack[0.3]{An infinite p19 oscillator}{\href{https://conwaylife.com/?rle=bo$2bo$3o3$5bo$6b2o$5b2o3$11bo$9bobo$10b2o3$14bobo$15b2o$15bo2$20bo$21bo$19b3o!&name=(A segment of) An infinite p19 oscillator}{\includegraphics[width=\textwidth]{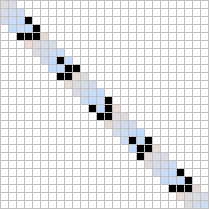}}}
\end{center}

We restrict our attention to \emph{finite} oscillators, because omniperiodicity for infinite oscillators is nearly trivial. An infinite stream of evenly-spaced gliders forms an oscillator, repeating when each glider reaches the position of the next. This gives an easy construction for all periods $p \geq 14$; the tightest the gliders can be packed without interfering. Oscillators with all remaining periods $p < 14$ were known by the mid-1970s.


\subsection{Handcrafted Oscillators, Billiard Tables and Hasslers}\label{sec:manual-searches}

\begin{center}
\patstack{(p3) Pulsar}{\href{https://conwaylife.com/?rle=2b3o3b3o2$o4bobo4bo$o4bobo4bo$o4bobo4bo$2b3o3b3o2$2b3o3b3o$o4bobo4bo$o4bobo4bo$o4bobo4bo2$2b3o3b3o!&name=(p3) pulsar}{\includegraphics[width=\textwidth]{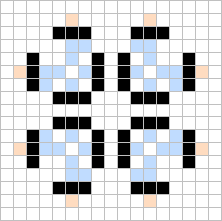}}}
\hfill
\patstack{(p5) Octagon 2}{\href{https://conwaylife.com/?rle=3b2o$2bo2bo$bo4bo$o6bo$o6bo$bo4bo$2bo2bo$3b2o!&name=(p5) octagon 2}{\includegraphics[width=\textwidth]{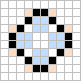}}}
\hfill
\patstack{(p8) Figure eight}{\href{https://conwaylife.com/?rle=3o$3o$3o$3b3o$3b3o$3b3o!&name=(p8) figure eight}{\includegraphics[width=\textwidth]{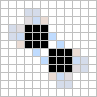}}}
\hfill
\patstack{(p14) Tumbler}{\href{https://conwaylife.com/?rle=2o3b2o$obobobo$obobobo$2bobo$b2ob2o$b2ob2o!&name=(p14) tumbler}{\includegraphics[width=\textwidth]{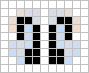}}}
\hfill
\patstack{(p15) \\ Pentadecathlon}{\href{https://conwaylife.com/?rle=2bo4bo$2ob4ob2o$2bo4bo!&name=(p15) pentadecathlon}{\includegraphics[width=\textwidth]{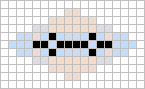}}}
\end{center}

The first few oscillators in Life were found through experimentation with small starting configurations, often evolved by hand on graph paper or covertly on corporate computers. The block (period 1) and blinker (p2) were found almost immediately, descending directly from the two triominoes. Other early oscillators were found from the lucky evolution of simple starting configurations, such as the pulsar (p3), figure eight (p8), and pentadecathlon (p15) found by Conway's team, the tumbler (p14) found by George D. Collins, Jr., and octagon 2 (p5) found independently by Sol Goodman and Arthur C. Taber \cite{gardner:1970, gardner:1971, lifeline:one}.

\begin{center}
\patstack{(p4) Pinwheel}{\href{https://conwaylife.com/?rle=6b2o$6b2o2$4b4o$2obo2bobo$2obobo2bo$3bo3b2ob2o$3bo4bob2o$4b4o2$4b2o$4b2o!&name=(p4) pinwheel}{\includegraphics[width=\textwidth]{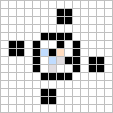}}}
\hfill
\patstack{(p6) \$rats}{\href{https://conwaylife.com/?rle=5b2o$6bo$4bo$2obob4o$2obo5bobo$3bo2b3ob2o$3bo4bo$4b3obo$7bo$6bo$6b2o!&name=(p6) \%24rats}{\includegraphics[width=\textwidth]{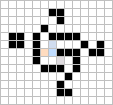}}}
\hfill
\patstack{(p7) Burloaferimeter}{\href{https://conwaylife.com/?rle=4b2o$5bo$4bo$3bob3o$3bobo2bo$2obo3bobo$2obobo2bo$4b4o2$4b2o$4b2o!&name=(p7) burloaferimeter}{\includegraphics[width=\textwidth]{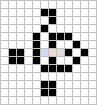}}}
\hfill
\patstack{(p10) 42P10.3}{\href{https://conwaylife.com/?rle=6b2o$6b2o2$4b4o$3bo4bo$o2b5obo$3o6bo$3bo3bobob2o$2bo2bo4bobo$2bob2ob3o$3bo2bo$4bo2bo$5b2o!&name=(p10) 42P10.3}{\includegraphics[width=\textwidth]{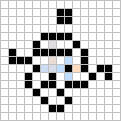}}}
\hfill
\patstack{(p11) 38P11.1}{\href{https://conwaylife.com/?rle=2b2ob2o$3bobobo$3bo4bo$2obo5bo$2obo6bo$3bobo5bo$3bob2o3b2o$4bo$5b7o$11bo$7b2o$7b2o!&name=(p11) 38P11.1}{\includegraphics[width=\textwidth]{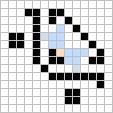}}}
\end{center}

In early 1970, while working as part of Conway's group, Simon Norton discovered a way to construct a class of oscillators that came to be known as \emph{billiard tables}. A rectangular region in the Life plane can be surrounded by a boundary of live cells and stabilized from the outside by still lifes. There is a chance that a configuration within the bounded region will then oscillate with aid from the boundary walls, interacting with but not destroying them. Norton used this technique to find pinwheel, the first period-4 oscillator. Billiard tables need not be restricted to rectangular regions and solid cell borders: throughout the 1970s Buckingham found billiard tables of varying structures including \$rats (p6), burloaferimeter (p7), and unnamed oscillators of periods 10 and 11, each the first known oscillator of their period \cite{hickerson:stamp-collection}.

\begin{center}
\patstack[0.3]{(p30) queen bee shuttle}{\href{https://conwaylife.com/?rle=9b2o$9bobo$4b2o6bo7b2o$2obo2bo2bo2bo7b2o$2o2b2o6bo$9bobo$9b2o!&name=(p30) queen bee shuttle}{\includegraphics[width=\textwidth]{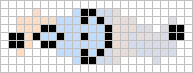}}}
\end{center}

Many larger periods are covered by another class of oscillators called \emph{hasslers}, which make up a majority of the new-period discoveries in the 21st century. These consist of an active pattern surrounded by objects that perturb its evolution and cause or allow it to reoccur. The surrounding objects are called \emph{catalysts} and are typically still lifes or other, smaller oscillators. Most hasslers involve active areas that evolve from commonly occurring patterns, as those are the most likely to reoccur naturally after being perturbed.

The first hassler used a pattern found by Robert W. April known as the \emph{queen bee}. After 15 generations, it reappears facing the opposite direction while leaving behind an extraneous still life known as a \emph{beehive}, visible on the left-hand side above. Advancing another 15 generations brings the queen bee back to its original position while depositing a second beehive, but subsequent evolution is disrupted when the advancing queen bee collides with the first beehive. Two queen bees can be positioned such that these extraneous beehives mutually annihilate, and Bill Gosper used this fact to form the first period-30 oscillator: a square of eight interacting queen bees~\cite{gosper:paleoballistics}.

The original form of this eight-fold queen bee oscillator was not recorded. Shortly after its construction, Richard P. Howell noticed that a properly placed block annihilates a beehive without itself being destroyed. This gives the much smaller queen bee shuttle (p30) shown above, with the blocks acting as catalysts.

\begin{center}
\patstack[0.3]{Gosper glider gun}{\href{https://conwaylife.com/?rle=27bo$26bobo$9b2o15b2obo$9bobo14b2ob2o3b2o$4b2o6bo13b2obo4b2o$2obo2bo2bo2bo13bobo$2o2b2o6bo8bo5bo$9bobo7bobo$9b2o9b2o5$28bo$29bo$27b3o!&name=Gosper glider gun}{\includegraphics[width=\textwidth]{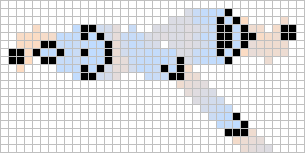}}}
\end{center}

The queen bee comprises one half of the famous \emph{Gosper glider gun}, the first pattern found that exhibits infinite growth~\cite{gardner:1971}. The two queen bees in this glider gun interact every 30 generations such that a glider is released from the mechanism.

\begin{center}
\patstack[0.3]{(p46) twin bees shuttle}{\href{https://conwaylife.com/?rle=17bo$2o15b2o8b2o$2o16b2o7b2o$13b2o2b2o4$13b2o2b2o$2o16b2o7b2o$2o15b2o8b2o$17bo!&name=(p46) twin bees shuttle}{\includegraphics[width=\textwidth]{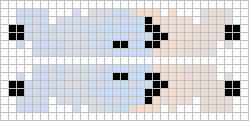}}}
\end{center}

In a similar fashion, a pair of ``B-heptominoes'' positioned with a certain spacing flips and reoccurs after 23 generations, leaving some junk behind. This junk can also be removed by placing blocks, forming the \emph{twin bees shuttle} (p46), also found by Bill Gosper (so named for its twin B's).

\begin{center}
\patstack[0.2]{(p9) worker bee}{\href{https://conwaylife.com/?rle=2o12b2o$bo12bo$bobo8bobo$2b2o8b2o2$5b6o2$2b2o8b2o$bobo8bobo$bo12bo$2o12b2o!&name=(p9) worker bee}{\includegraphics[width=\textwidth]{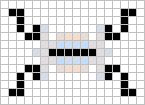}}}
\hfill
\patstack[0.15]{(p12) dinner table}{\href{https://conwaylife.com/?rle=bo$b3o7b2o$4bo6bo$3b2o4bobo$9b2o2$5b3o$5b3o$2b2o$bobo4b2o$bo6bo$2o7b3o$11bo!&name=(p12) dinner table}{\includegraphics[width=\textwidth]{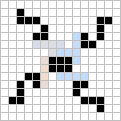}}}
\hfill
\patstack[0.15]{(p52) 35P52}{\href{https://conwaylife.com/?rle=2bo$2b3o$5bo9b2o$4b2o9bo$13bobo$13b2o3$6b3o$6bo2b2o$6b2o2bo$2b2o4b3o$bobo$bo9b2o$2o9bo$12b3o$14bo!&name=(p52) 35P52}{\includegraphics[width=\textwidth]{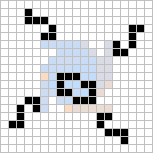}}}
\end{center}

While the queen bee and twin bees miraculously reappear without any external intervention, most hasslers do not work this way. An example is the worker bee (p9), discovered by Buckingham in 1972~\cite{lifeline:five}, where a line of six live cells evolves briefly before having its corners perturbed by four copies of a 7-cell catalyst known as an \emph{eater}. The catalysis involves the corner of the eater being briefly disturbed while the rest of the still life remains unchanged. This robust ``eating'' capability was first noted by researchers at MIT in 1971~\cite{lifeline:three} and quickly became a valuable tool in the construction of new hasslers, two more of which are displayed above~\cite{lifeline:six, hickerson:stamp-collection}.

\begin{center}
~
\hfill
\patstack[0.25]{Buckingham's p13}{\href{https://conwaylife.com/?rle=4bo15bo$3bobo13bobo$3bobo13bobo$b3ob2o11b2ob3o$o23bo$b3ob2o11b2ob3o$3bob2o11b2obo$10bo3bo$11bobo$10b2ob2o$8bo2bobo2bo$7bobobobobobo$7b2o2bobo2b2o$11bobo$12bo!&name=Buckingham's p13}{\includegraphics[width=\textwidth]{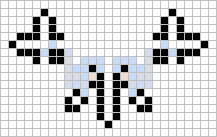}}}
\hfill
\patstack[0.2]{(p32) gourmet}{\href{https://conwaylife.com/?rle=10b2o$10bo$4b2ob2obo4b2o$2bo2bobobo5bo$2b2o4bo8bo$16b2o2$16b2o$o9b3o2bobo$3o7bobo3bo$3bo6bobo4b3o$2bobo14bo$2b2o2$2b2o$2bo8bo4b2o$4bo5bobobo2bo$3b2o4bob2ob2o$9bo$8b2o!&name=(p32) gourmet}{\includegraphics[width=\textwidth]{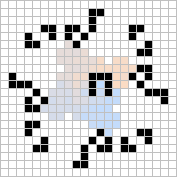}}}
\hfill
~
\end{center}

Many other stable patterns can act as catalysts. Buckingham in particular possessed an uncanny ability to construct complex catalysts by hand. Of note is the \emph{eater 2}~\cite{byte:facts-of-life} which catalyzes reactions in a similar way to the eater, but takes longer to recover. This extended recovery time allows it to catalyze some reactions that the eater cannot, such as in Buckingham's p13 oscillator, the first of its period known. Another early hassler using an unconventional catalyst is the period-32 gourmet found by Buckingham in 1978~\cite{hickerson:stamp-collection}, which rotates a pi-heptomino by 90 degrees every 8 generations.

\begin{center}
~
\hfill
\patstack[0.2]{(p16) Two pre-L hassler}{\href{https://conwaylife.com/?rle=4b2o12b2o$2obo2bob2o4b2obo2bob2o$2o2bo4bo4bo4bo2b2o$5bo12bo$6bobo6bobo$13bo$12bobo$9bo5bo$8bo5bo$9bobo$10bo$6bobo6bobo$5bo12bo$2o2bo4bo4bo4bo2b2o$2obo2bob2o4b2obo2bob2o$4b2o12b2o!&name=(p16) Two pre-L hassler}{\includegraphics[width=\textwidth]{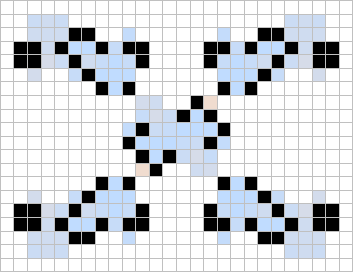}}}
\hfill
\patstack[0.2]{(p18) 117P18}{\href{https://conwaylife.com/?rle=37b2o$8b2o18b2o7bo$9bo7b2o7bo4bo3bobo$9bobo3bo4bo4bo6bo2b2o$10b2o2bo6bo2bo8bo$13bo8bobo8bo$13bo8bobo8bo$13bo8bo2bo6bo2b2o$10b2o2bo6bo4bo4bo3bobo$9bobo3bo4bo7b2o7bo$9bo7b2o18b2o$8b2o$19bo3b2o$17b2ob2ob2o$17bo2bo$18b2o$14b2o$14b2o$29b2o$2o18b2o7bo$bo7b2o7bo4bo3bobo$bobo3bo4bo4bo6bo2b2o$2b2o2bo6bo2bo8bo$5bo8bobo8bo$5bo8bobo8bo$5bo8bo2bo6bo2b2o$2b2o2bo6bo4bo4bo3bobo$bobo3bo4bo7b2o7bo$bo7b2o18b2o$2o!&name=(p18) 117P18}{\includegraphics[width=\textwidth]{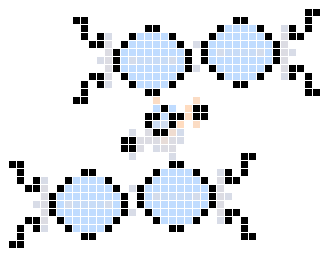}}}
\hfill
\patstack[0.2]{(p36) toad hassler}{\href{https://conwaylife.com/?rle=2o16b2o$bo7b2o7bo$bobo3bo4bo3bobo$2b2o2bo6bo2b2o$5bo8bo$5bo8bo$5bo8bo$2b2o2bo6bo2b2o$bobo3bo4bo3bobo$bo7b2o7bo$2o16b2o$11bo$9bo2bo$9bo2bo$10bo$2o16b2o$bo7b2o7bo$bobo3bo4bo3bobo$2b2o2bo6bo2b2o$5bo8bo$5bo8bo$5bo8bo$2b2o2bo6bo2b2o$bobo3bo4bo3bobo$bo7b2o7bo$2o16b2o!&name=(p36) toad hassler}{\includegraphics[angle=90,width=\textwidth]{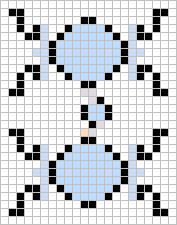}}}
\hfill
~
\end{center}

Catalysts need not be still lifes. Many hasslers are catalyzed by low-period oscillators that emit \emph{sparks} --- small groups of active cells that are ejected from the main body of an oscillator before dying off. Such \emph{sparkers} are powerful as catalysts because the sparks are free to interact with the active pattern in many different ways. Because they would die anyway, they do not need to be restored for the sparker to continue to function. 
When a sparker is used as a catalyst, the period of the sparker must divide the period of the entire oscillator so that the spark appears at the right time during the next period.

Very few sparkers were known prior to 1989, so hasslers exploiting them were less common. Above we have displayed three such oscillators, on the generation just as the spark interacts with the active pattern. Each was the first of its period known.


\begin{center}
\patstack[0.25]{(p26) pre-pulsar shuttle}{\href{https://conwaylife.com/?rle=16b2o3b2o$15bo2bobo2bo$11b2o3b2o3b2o3b2o$11bo15bo$8b2obo15bob2o$7bobob2o13b2obobo$7bobo5b3o3b3o5bobo$5b2o2bo5bobo3bobo5bo2b2o$4bo4b2o4b3o3b3o4b2o4bo$4b5o21b5o$8bo21bo$2b4o27b4o$2bo2bo27bo2bo2$16b2o3b2o$bo4b3o8bo3bo8b3o4bo$obo3bobo5bo9bo5bobo3bobo$obo3b3o5b2o7b2o5b3o3bobo$bo35bo2$bo35bo$obo3b3o5b2o7b2o5b3o3bobo$obo3bobo5bo9bo5bobo3bobo$bo4b3o8bo3bo8b3o4bo$16b2o3b2o2$2bo2bo27bo2bo$2b4o27b4o$8bo21bo$4b5o21b5o$4bo4b2o4b3o3b3o4b2o4bo$5b2o2bo5bobo3bobo5bo2b2o$7bobo5b3o3b3o5bobo$7bobob2o13b2obobo$8b2obo15bob2o$11bo15bo$11b2o3b2o3b2o3b2o$15bo2bobo2bo$16b2o3b2o!&name=p26 pre-pulsar shuttle}{\includegraphics[width=\textwidth]{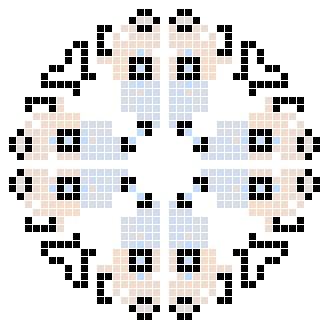}}}
\hfill
\patstack[0.25]{(p28) newshuttle}{\href{https://conwaylife.com/?rle=26b2o$20bo3bo2bo2bo$18b3o3b3o3b3o$8b2o7bo15bo7b2o$9bo7b2o5b3o5b2o7bo$9bobo11bo3bo11bobo$10b2o11b2ob2o11b2o2$3bo43bo$3b3o39b3o$6bo7b3o17b3o7bo$5b2o7bobo17bobo7b2o$14b3o3b3o5b3o3b3o$20bobo5bobo$10b3o7b3o5b3o7b3o$10bobo25bobo$10b3o25b3o$3b2o18b2ob2o18b2o$2bobo18bo3bo18bobo$2bo21b3o21bo$b2o9b3o11bobo7b3o9b2o$12bobo12b2o7bobo$12b3o5b2o14b3o$2o3b2o10b2o2bo10b2o10b2o$obobobo10bob2o10bobo10bobob2o$2bobo14bo11bo14bobo$b2obobo10bobo10b2obo10bobobobo$5b2o10b2o10bo2b2o10b2o3b2o$12b3o14b2o5b3o$12bobo7b2o12bobo$b2o9b3o7bobo11b3o9b2o$2bo21b3o21bo$2bobo18bo3bo18bobo$3b2o18b2ob2o18b2o$10b3o25b3o$10bobo25bobo$10b3o7b3o5b3o7b3o$20bobo5bobo$14b3o3b3o5b3o3b3o$5b2o7bobo17bobo7b2o$6bo7b3o17b3o7bo$3b3o39b3o$3bo43bo2$10b2o11b2ob2o11b2o$9bobo11bo3bo11bobo$9bo7b2o5b3o5b2o7bo$8b2o7bo15bo7b2o$18b3o3b3o3b3o$20bo2bo2bo3bo$23b2o!&name=(p28) newshuttle}{\includegraphics[width=\textwidth]{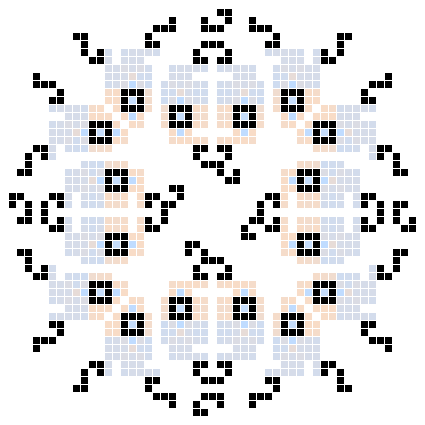}}}
\hfill
\patstack[0.25]{(p29) pre-pulsar shuttle}{\href{https://conwaylife.com/?rle=15bo$13b3o$12bo$12b2o2$bo$obo6b3o$bo7bobo3b2o$9b3o3b2o3$19b2o$9b3o7b2o5b2o$bo7bobo14bo$obo6b3o12bobo$bo22b2o$13b3o3b3o$13bobo3bobo$13b3o3b3o7$13bo7bo$12bobo5bobo$13bo7bo!&name=p29 pre-pulsar shuttle}{\includegraphics[width=\textwidth]{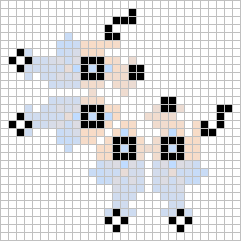}}}
\end{center}

A notable class of early hasslers pioneered by Buckingham uses a common predecessor of the pulsar. The \emph{pre-pulsar} is an arrangement of two hollow $3\times 3$ squares that interact in such a way that they duplicate themselves after 15 generations. Various catalyses can be used to hasten or delay this reaction, and combined they can be used to create oscillators of many periods. The three depicted above, all by Buckingham, were the first known oscillators of their respective periods.

\begin{center}
\patstack[0.3]{(p23) David Hilbert}{\href{https://conwaylife.com/?rle=7b2o15b2o$8bo15bo$6bo19bo$6b5o11b5o$10bo11bo$4b4o16b5o$4bo2bo17bo2bo$7bo2bo14b2o$7bo3bo12b3o$7bo3bo2b2ob2o5b2o$8bo2b2ob2ob2o5b2o$10bo12bo$8b2o$7b3o13bo$3b2o2b3o6b2o4bobo3b2o$3bo11bo2bo3bobo4bo$2obo12b2o4b3o4bob2o$ob2ob2o19b2ob2obo$5bo21bo$5bobo17bobo$6b2o17b2o$10bo11bo$6b5o11b5o$6bo19bo$8bo15bo$7b2o15b2o!&name=(p23) David Hilbert}{\includegraphics[width=\textwidth]{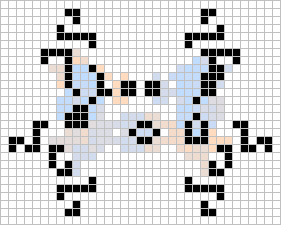}}}
\end{center}

The use of computer searches did not foreclose the ability of humans to find new oscillators by hand through experimentation and intuition. A final hand-crafted hassler worthy of mention is the first known period-23 oscillator: this was the 5th last period to be filled. An attempt from 2015 by forum user ``praosylen'' was modified by Luka Okanishi in 2019 to be a near-oscillator; the only defect being that a beehive is pushed away by a one-cell distance~\cite{p23-completion}. To push the beehive back, a second copy of the near-oscillator is used, running out-of-phase with the first.

\subsection{Soup Searches}\label{sec:soup-searches}

\begin{center}
~
\hfill
\patstack[0.2]{(p3) jam}{\href{https://conwaylife.com/?rle=4b2o$3bo2bo$bo2bobo$o4bo$o$2b3o!&name=(p3) jam}{\includegraphics[width=\textwidth]{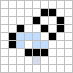}}}
\hfill
\patstack[0.2]{(p4) mold}{\href{https://conwaylife.com/?rle=3b2o$2bo2bo$bobobo$bo2bo$o$bobo!&name=(p4) mold}{\includegraphics[width=\textwidth]{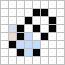}}}
\hfill
\patstack[0.2]{Achim's p11}{\href{https://conwaylife.com/?rle=10b2o$10b2o$5bo10bo$4bobo8bobo$3bobo10bobo$2bobo3b6o3bobo$3bo14bo2$5bo10bo$5bo10bo$2o3bo10bo3b2o$2o3bo10bo3b2o$5bo10bo$5bo10bo2$3bo14bo$2bobo3b6o3bobo$3bobo10bobo$4bobo8bobo$5bo10bo$10b2o$10b2!&name=Achim's p11}{\includegraphics[width=\textwidth]{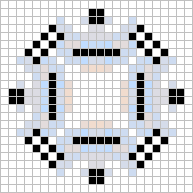}}}
\hfill
~
\end{center}

As we have seen, one method for finding oscillators is to start with a random configuration and get lucky. Done on a large scale, these are known as \emph{soup searches}, essentially automated versions of the random scribbling that newcomers to Life perform. A small random starting configuration (i.e., a \emph{soup}) is chosen, optionally with symmetry, and then run until it stabilises. The objects that remain in the ash are programmatically separated from one another and tabulated, resulting in a \emph{census}. With carefully optimised software, tens of thousands of soups can be tested per second on a standard computer.

From 1987 to 1988, Achim Flammenkamp wrote the first soup-searching programs, running soups on a torus --- a finite universe that wraps around on its edges. The period-3 \emph{jam} and the period-4 \emph{mold} were discovered this way, spotted in the ash of soups run on 32\texttimes{}32 toruses. A soup search on symmetrical configurations yielded the sparking p11 oscillator on the right, which makes an appearance in Section~\ref{sec:period-multipliers}.

\begin{center}
~
\hfill
\patstack[0.2]{Rob's p16}{\href{https://conwaylife.com/?rle=4b2o2$4b3o$2bob2o$b2obo$b2ob2o$bo6bo$2o3bobobo$bo3bo2bo!&name=Rob's p16}{\includegraphics[width=\textwidth]{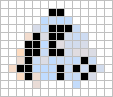}}}
\hfill
\patstack[0.2]{Rich's p16}{\href{https://conwaylife.com/?rle=3b3o3b3o2$b2o3bobo3b2o$bo4bobo4bo$o5bobo5bo$bo4bobo4bo$2b3o5b3o2$5b2ob2o$4bobobobo$5bo3bo!&name=Rich's p16}{\includegraphics[width=\textwidth]{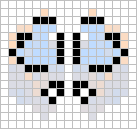}}}
\hfill
\patstack[0.2]{(p25) 30P25}{\href{https://conwaylife.com/?rle=o$3o$3bo$2b2o2$6bo$6b2o$7b2o$4b2o$4b3o$13b3o$14b2o$11b2o$12b2o$13bo2$16b2o$16bo$17b3o$19bo!&name=(p25) 30P25}{\includegraphics[width=\textwidth]{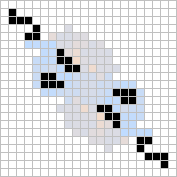}}}
\hfill
~
\end{center}

Several large censuses have been run since then \cite{flammenkamp:census, okrasinski:census, tollcass:census}, but by far the largest is the ongoing Catagolue~\cite{catagolue} project run by Adam P. Goucher, having searched more than 450 trillion soups and tallied more than 12 quadrillion Life objects since 2015 (at the time of writing). In 2022, the Catagolue search was connected to the Charity Engine and OpenScienceGrid~\cite{osg07} distributed computing platforms, significantly boosting the search speed. 

Oscillators that are found via soup searches are often spectacularly volatile, with few cells that remain stable through the entire period. This is in contrast to the billiard tables and hasslers we have seen, where the majority of the alive cells are stable, only present to support a small active region. Soup searches are essentially the only method we have of finding small, volatile, high-period oscillators.

Above is displayed Rob's p16 (found by Rob Liston), Rich's p16 (found by Rich Holmes) and 30P25 (found by Charity Engine), all through the Catagolue project.

\begin{center}
\patstack[0.2]{Merzenich's p31}{\href{https://conwaylife.com/?rle=7b2obo2bob2o$2o4bo2bo4bo2bo4b2o$2o5bobo4bobo5b2o$8bo6bo6$8bo6bo$2o5bobo4bobo5b2o$2o4bo2bo4bo2bo4b2o$7b2obo2bob2o!&name=Merzenich's p31}{\includegraphics[width=\textwidth]{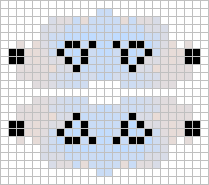}}}
\hfill
\patstack[0.2]{Beluchenko's p37}{\href{https://conwaylife.com/?rle=11b2o11b2o$11b2o11b2o3$6bo23bo$5bobo6b3o3b3o6bobo$4bo2bo5bo9bo5bo2bo$5b2o6bo2bo3bo2bo6b2o$14b2o5b2o3$2o33b2o$2o33b2o$6b2o21b2o$5bo2bo19bo2bo$5bo2bo19bo2bo$5bobo21bobo4$5bobo21bobo$5bo2bo19bo2bo$5bo2bo19bo2bo$6b2o21b2o$2o33b2o$2o33b2o3$14b2o5b2o$5b2o6bo2bo3bo2bo6b2o$4bo2bo5bo9bo5bo2bo$5bobo6b3o3b3o6bobo$6bo23bo3$11b2o11b2o$11b2o11b2o!&name=Beluchenko's p37}{\includegraphics[width=\textwidth]{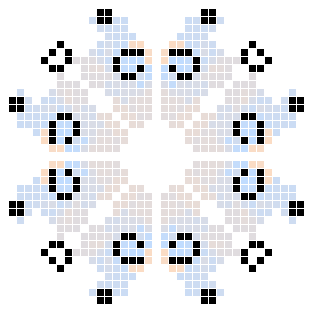}}}
\hfill
\patstack[0.2]{Beluchenko's p51}{\href{https://conwaylife.com/?rle=3b2o21b2o$3b2o21b2o2$2o27b2o$2o5bo2b3o5b3o2bo5b2o$6b3ob3o5b3ob3o$5bo2bo13bo2bo$4b2o19b2o$5b2o17b2o2$4b2o19b2o$4b2o19b2o$4b2o19b2o6$4b2o19b2o$4b2o19b2o$4b2o19b2o2$5b2o17b2o$4b2o19b2o$5bo2bo13bo2bo$6b3ob3o5b3ob3o$2o5bo2b3o5b3o2bo5b2o$2o27b2o2$3b2o21b2o$3b2o21b2o!&name=Beluchenko's p51}{\includegraphics[width=\textwidth]{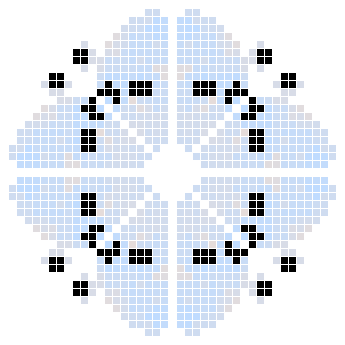}}}
\end{center}

\texttt{torus}~\cite{summers:torus} by Jason Summers and \texttt{RandomAgar}~\cite{randomagar} by Gabriel Nivasch are soup search programs that attempt to find \emph{agars}: patterns that are periodic in space as well as periodic in time. They do so by choosing a symmetry (one of the wallpaper groups compatible with a square grid\footnote{Of the 17 wallpaper groups, 5 contain a rotation of 120$^{\circ}$ and so cannot be used. Of the 12 remaining, 9 yield distinct symmetries when aligned to the grid orthogonally and diagonally, so there are 21 possibilities in all.}) and size of a fundamental region, filling the fundamental region with a random initial configuration, and simulating the finite universe until it becomes periodic. In some cases, a complete oscillator will occur. The first oscillators with periods 31, 37 and 51, depicted above, were found in this way using a variant of \texttt{RandomAgar} written by Nicolay Beluchenko~\cite{lifenews:p31, lifenews:p37, lifenews:p51}.

\begin{center}
\hfill
\patstack[0.2]{(p36) 22P36}{\href{https://conwaylife.com/?rle=17bo$15b3o$4b3o7bo$14b2o$2bo5bo$2bo5bo2b3o$2bo5bo$9bo5bo$4b3o2bo5bo$9bo5bo$2b2o$3bo7b3o$3o$o!&name=(p36) 22P36}{\includegraphics[width=\textwidth]{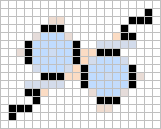}}}
\hfill
\patstack[0.4]{Raucci's p38}{\href{https://conwaylife.com/?rle=4bo32bo$4bo32bo$4bo32bo$6b2o3bo18bo3b2o$3o3b2o3b2o16b2o3b2o3b3o$13bo14bo$12b2o14b2o3$6b2o26b2o$6bo28bo$7b3o22b3o$9bo22bo!&name=Raucci's p38}{\includegraphics[width=\textwidth]{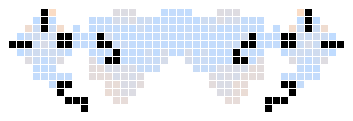}}}
\hfill
~
\end{center}

In other cases the agars do not contain complete oscillators, but external ``fenceposts'' can be found to replace the influence of the wraparound at the edges of the fundamental region, resulting in a standalone oscillator.  A notable oscillator that was found this way is Raucci's p38: this period was the fourth last to be filled. The active part of the oscillator was found by Jason Summers in 2000, as an infinite agar. Only 22 years later, some appropriate fenceposts were found: a catalyst known as the \emph{traffic stop}. This catalyst was discovered by a user known pseudonymously as ``cvojan'', and can be seen as one half of a period-36 oscillator found by Noam D. Elkies in 1995. It was then applied by David Raucci to Summers' period-38 agar---an excellent example of how much work and collaboration goes into a single result!

\subsection{Direct Searches and Drifters}

The simplest strategy for searching directly for oscillators is a brute force search on all states of a fixed size. This quickly becomes unfeasible and so more refined strategies are needed.

The earliest software capable of searching for oscillators directly was written by Dean Hickerson in 1989. The program performs a depth-first search on the state of each cell in a finite grid in every phase of the desired period, where the rules of Life are enforced forwards and backwards in time. This puts strong constraints on which choices for each cell are allowed: as soon as any cell is known to not oscillate at the correct period, that branch of the search tree can be dropped. This approach is generally effective only for finding low-period oscillators, and early results were further limited by the speed of Hickerson's Apple II.
Hickerson described the method in detail, and several modern search tools are descendants of that original program~\cite{winlifesearch, javalifesearch}, most notably \texttt{lifesrc}~\cite{lifesrc} by David Bell.

\begin{center}
\hfill
\patstack[0.25]{(p25) 134P25}{\href{https://conwaylife.com/?rle=11b2o$7b2o2bo$6bo2bobo$2b2obob3ob2o$2bob2obobobo$13bo$3b4obob3obo$2bo3b2ob2o3bo$3b3ob3ob3o$5bo5bo7bo$17b3o$6bo3bo5bo$8bo7b2o$2o$bo$bobo9b3o$2b2o9bobo$7b3o3b3o$7bobo9b2o$7b3o9bobo$21bo$21b2o$5b2o7bo$6bo5bo3bo$3b3o$3bo7bo5bo$9b3ob3ob3o$8bo3b2ob2o3bo$8bob3obob4o$9bo$11bobobob2obo$10b2ob3obob2o$11bobo2bo$11bo2b2o$10b2o!&name=(p25) 134P25}{\includegraphics[angle=90,width=\textwidth]{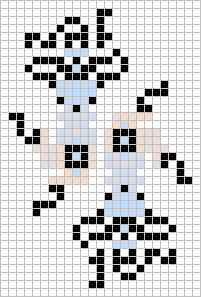}}}
\hfill
\patstack[0.2]{(p40) B-heptomino shuttle}{\href{https://conwaylife.com/?rle=b2o20b2o$b2o20bobo$24b3o$25b2o$3o19b2o$2o20b3o$3b2o$2b3o$bobo3b2o2b2o10b2o$b2o3bo8b2o6b2o$7b2o7b2o$15b2o$11bo3bo$11bo$7bo4bo5bo$6bobo8bobo$7bo10bo$11bo2bo$9bo2b2o2bo$9bo2b2o2bo$11b4o$9bobo2bobo$9b2o4b2o!&name=(p40) B-heptomino shuttle}{\includegraphics[width=\textwidth]{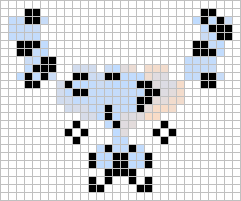}}}
\hfill
~
\end{center}

After the late 1980s, these cell-by-cell search programs allowed more powerful sparkers to be found. The first known oscillators of periods 25 and 40 use some of these newer sparkers.

\begin{center}
\patstack{(p17) 54P17.1}{\href{https://conwaylife.com/?rle=5bo$4bobo$4bobo3b2o$b2obob2o3bo$2bobo6bob2o$o2bob2ob2obo2bo$2obobobo2bobo$3bob2ob2o2b2o$3bobo3bobo$4b2obobobo$6bobobo$6bobo$7b2o!&name=(p17) 54P17.1}{\includegraphics[width=\textwidth]{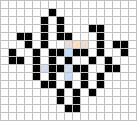}}}
\end{center}

The search program \texttt{dr}~\cite{dr} by Hickerson is similar to these cell-by-cell search programs, but makes a strong assumption that turns out to be extremely effective in practice: that the active part of the pattern is a small, unstable region present on a completely stable background. This could be seen as a systematic method of hunting for new billiard tables.

The program was originally written to search for signals traveling through stable \emph{wire}, after the manual discovery of a signal wire with speed $2c/3$\footnote{In this context, $c$ is the speed of one cell per generation, the maximum speed attainable by any signal in Life.}. The hope was that an \emph{elbow} could be found, turning a signal 90 degrees. If the repeat time of this elbow were short enough (i.e. $\leq 19$), this would immediately resolve omniperiodicity in the affirmative. These searches were not successful, though there were some near misses~\cite{lexicon:elbow}, including an alternative signal wire with transmission speed $5c/9$.

The \texttt{dr} program can also identify when it has stumbled on an oscillator. Typically, these have low periods, but the result can occasionally get into the high teens and low twenties~\cite{hickerson:new-billard-tables}. The first period-17 oscillator was found this way by Hickerson. This general strategy of using a backtracking search to find still lifes with good properties can also be used to produce powerful catalysts for use in hasslers, as we will see soon. 

\subsection{Period Multipliers}\label{sec:period-multipliers}

\begin{center}
\patstack{A trivial p12}{\href{https://conwaylife.com/?rle=8bobo$11bo$7bo2bo$6bobobo$6bo2bo$4b2ob2o$3bo2bo$bo2bobo$o4bo$o$2b3o!&name=A trivial p12}{\includegraphics[width=\textwidth]{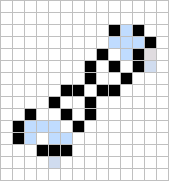}}}
\end{center}

From oscillators of period $n$ and $m$, we can easily construct a pattern with period $\lcm(n, m)$ by placing the oscillators in the plane such that they do not interact. These are \emph{trivial} period-$\lcm(n, m)$ oscillators, as no cell of the pattern oscillates at the full period. Above, we show a trivial p12 oscillator constructed by welding the p3 jam and p4 mold back-to-back.

\begin{center}
\patstack{(p20)}{\href{https://conwaylife.com/?rle=3bo2bo$3bo2bo$b2ob2ob2o$3bo2bo$3bo2bo$b2ob2ob2o$3bo2bo$3bo2bo2$7b2o$2b2obo4bob2o$2bo10bo$3b2o6b2o$3o2b6o2b3o$o2bo8bo2bo$b2o10b2o!&name=(p20)}{\includegraphics[width=\textwidth]{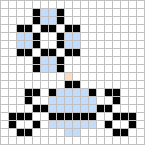}}}
\hfill
\patstack{(p21)}{\href{https://conwaylife.com/?rle=10b2o$4bo4bo2bo$3bobo3bobo$3bobo4bo3b2o$2obob2o4b2obo$ob2o4bo3bo$4b3obo3bo$4bo$5b3o2b2o$8bo$5b2obobobo$5b2obobob3o$9b2o4bo$11b4o$11bo$12bo$11b2o!&name=(p21)}{\includegraphics[width=\textwidth]{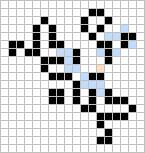}}}
\hfill
\patstack{(p24)}{\href{https://conwaylife.com/?rle=4bo5bo$4bo5bo7b2o$4b2o3b2o$16bo3bo$3o2b2ob2o2b3obo4bo$2bobobobobobo5bobobo$4b2o3b2o8bobobo$20bo4bo$4b2o3b2o10bo3bo$2bobobobobobo$3o2b2ob2o2b3o7b2o2$4b2o3b2o$4bo5bo$4bo5bo!&name=(p24)}{\includegraphics[width=\textwidth]{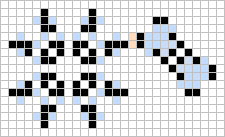}}}
\hfill
\patstack{(p35)}{\href{https://conwaylife.com/?rle=16bo6b2o$15bobob2o2bo$4bo10bobobobobo$3bobo8b2obo3bob2o$3bobo6bo3b2obob2o3bo$2obob2o8bo7bo$ob2o4bo6bo7bo$4b3obo3bo3b2obob2o3bo$4bo9b2obo3bob2o$5b3o2b2o3bobobobobo$8bo6bobob2o2bo$5b2obobobo3bo6b2o$5b2obobob4o$9b2o$11b3o$11bo2bo$13b2o!&name=(p35)}{\includegraphics[width=\textwidth]{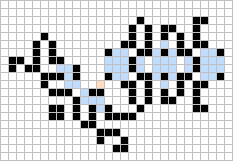}}}
\hfill
\patstack{(p42)}{\href{https://conwaylife.com/?rle=16b2o$16bo2bo2$4bo$3bobo10bob2o$3bobo5b2o2bobo$2obob2o4bo4bo$ob2o4bo6bo$4b3obo3bo2bo$4bo$5b3o2b2o$8bo$5b2obobobo$5b2obobob3o$9b2o4bo$11b4o$11bo$12bo$11b2o!&name=(p42)}{\includegraphics[width=\textwidth]{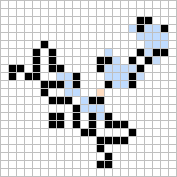}}}
\end{center}

Sometimes, however, two oscillators can be placed such that they just barely interact without destroying either oscillator; such patterns are called \emph{LCM oscillators}. One way to do this uses sparkers. For example, the first period-20 oscillator consists of a period-4 oscillator that emits domino sparks interacting with the period-5 octagon. When the domino spark and the edge of the octagon 2 are present on the same generation, they interact and turn on a single cell between the two oscillators. The first oscillators of periods $20 = 4 \times 5$, $21 = 3 \times 7$, $24 = 3 \times 8$, $33 = 3 \times 11$, $35 = 5 \times 7$ and $42 = 6 \times 7$ were all found this way.

The above images present the oscillators on the generation before they interact; the interaction point is visible in orange. These interaction cells are the only ones that oscillate at the full period.

In fact, this method of period-multiplying could have been used to construct some periods \emph{earlier} than their official discovery dates: the requisite components existed at the time but were never combined. The periods in this category are 18, 36 and 40. For example, the octagon 2 and the figure eight may be combined to form a non-trivial oscillator of period 40 using the same interaction that appears in the period-20 oscillator above.

\begin{center}
\patstack[0.4]{(p33) 330P33}{\href{https://conwaylife.com/?rle=19b2o$19b2o$14bo10bo$13bobo8bobo$12bo2bo3b2o3bo2bo$11bo6b4o6bo$12b2o4bo2bo4b2o3$14b2o8b2o$9b2o2b2o10b2o2b2o$9b2o2b2o10b2o2b2o$14b2o8b2o3$12b2o4bo2bo4b2o$11bo6b4o6bo$12bo2bo3b2o3bo2bo$13bobo8bobo$14bo10bo$19b2o$14bo4b2o3$3bo3bobob6obobo3bo$2bob2obo2bobo2bobo2bob2obo$2bo7b2o4b2o7bo$b2o2bo3bo2b4o2bo3bo2b2o$o2b6ob2o4b2ob6o2bo$2o6b2ob6ob2o6b2o$2b2obo16bob2o$2o2bo8b2o8bo2b2o$bobo2b3o3bo2bo3b3o2bobo$o2b3o5b2o2b2o5b3o2bo$b2o3bob2obo4bob2obo3b2o$2bobo2bob3o4b3obo2bobo$2bob2obo3b6o3bob2obo$b2o2bobobobo4bobobobo2b2o$o2bo7b6o7bo2bo$b2o3bo14bo3b2o$5bo3b4o2b4o3bo$4bob2obo2bo2bo2bob2obo$4bo2b3o8b3o2bo$b2obo2b3o2bo2bo2b3o2bob2o$b2obob2o4b4o4b2obob2o$5bo3bo8bo3bo$6bobo10bobo$4bobobobo6bobobobo$4b2o3b2o6b2o3b2o!&name=(p33) 258P3 on Achim's p11}{\includegraphics[angle=90,width=\textwidth]{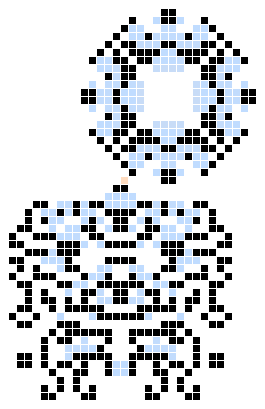}}}
\end{center}

The $33 = 11 \times 3$ oscillator involves a large sparker on the right, which was found by Elkies using \texttt{lifesrc}. It was the first known period-3 sparker to emit a domino spark so isolated from the remainder of the sparker, and has various other applications in Life.

\begin{center}
\patstack[0.3]{(p39) 134P39.1}{\href{https://conwaylife.com/?rle=7b2o3b2o$7bo3bobo$4b2obo2bo2bobo$4bo2b2o3bo2b3o$6bo3bo2bo4bo$7b3obob2ob3o$9bo5bo$10bo4bob2o$4b2o5b2ob2ob2o$2obobo3bo2bobo$2obob3ob2obobo10bo$3bo2b2o3bobo10bobo$3bobo2b2o2bo11bobo$4bobo3b2o11b2ob2o$6bobo$6bobob2o11b2ob2o$7b2ob2o11b2obo$15bo3bo8bo$16bobo8b2o$15b2ob2o$13bo2bobo2bo$12bobobobobobo$12b2o2bobo2b2o$16bobo$17bo!&name=(p39) 134P39.1}{\includegraphics[width=\textwidth]{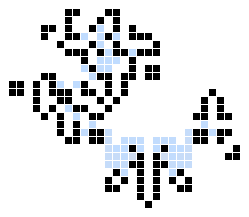}}}
\end{center}

The first period-39 oscillator also multiplies the period of two smaller oscillators, but uses a different mechanism to do so. The lower part of the oscillator is Buckingham's p13 (seen in Section \ref{sec:manual-searches}), which is normally supported by two stable eater 2s. In 2000, Elkies found a period-3 variant of the eater 2, such that when the eater 2 catalysis occurs, it toggles the state of an otherwise stable cell. The period-3 component then toggles the cell state back, in this case creating an oscillator with period $39 = 3 \times 13$.

\begin{center}
\patstack[0.2]{(p27) 123P27.1}{\href{https://conwaylife.com/?rle=b2ob2o$2bobobob2o$bo4bobo$ob4o4bo$o4bo3b4obo$b4o3bo4b2o$3bo4b4o3b2o$5bobo4b2o3bo$4b2obo2b2o2b2obo$3bo3bobo2bobobo$4b3o2b2obo6bo$7b2o3bo2b5o$6bo2bob2obobo$5bob2o3bo2bo2b2o$6bo2b2o3bobo2bo$7b2o3bobob2o$9b4obobo$9bo2bobobo$10b5ob2o$15bo2bo$9bob4o2b2o$9b2obo!&name=(p27) 123P27.1}{\includegraphics[width=\textwidth]{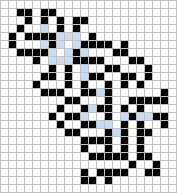}}}
\end{center}

A third technique is seen in the first known period-27 oscillator. When two oscillating mechanisms interact, one can cause the other to skip ahead or fall behind by one or more generations, a phenomenon known as \emph{phase-shifting}. The first period-27 oscillator, found by Elkies, is a period-9 oscillator with a small period-7 section attached. When the phases of the mechanisms align in a certain way, they interact such that the period-7 component skips one generation. 

\subsection{Signal Loops}

\begin{center}
  \patstack[0.25]{(p180) glider loop}{\href{https://conwaylife.com/?rle=22b2o$22bo9b2o$12b2o6bobo9b2o$11bobo5b3o$10bo6b3o13bo$b2o7bo2bo2bo2bo12bobo$b2o7bo6b2o13bobo$11bobo19bo$12b2o2$30b2obob2o$30bo5bo$31bo3bo$32b3o2$o$3o$3bo$2b2o2$35b2o$35bo$36b3o$38bo2$4b3o$3bo3bo$2bo5bo$2b2obob2o8bo$18bo$16b3o6b2o$5bo19bobo$4bobo13b2o6bo7b2o$4bobo12bo2bo2bo2bo7b2o$5bo13b3o6bo$17b3o5bobo$5b2o9bobo6b2o$5b2o9bo$15b2o!&name=(p180) glider loop}{\includegraphics[width=\textwidth]{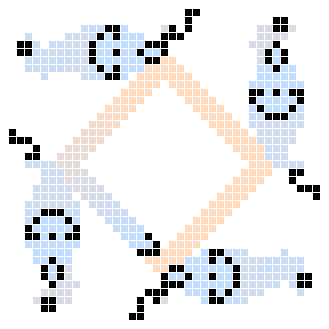}}}
  \hfill
  \patstack[0.25]{(p30) glider loop}{\href{https://conwaylife.com/?rle=22b2o$22bo9b2o$12b2o6bobo9b2o$11bobo5b3o$10bo6b3o13bo$b2o7bo2bo2bo2bo12bobo$b2o7bo6b2o13bobo$11bobo19bo$12b2o6b3o$20bo$21bo8b2obob2o$30bo5bo$13bo17bo3bo$11b2o19b3o$12b2o$o27b2o$3o25bobo$3bo24bo$2b2o2$35b2o$10bo24bo$8bobo25b3o$9b2o27bo$25b2o$4b3o19b2o$3bo3bo17bo$2bo5bo$2b2obob2o8bo$18bo$16b3o6b2o$5bo19bobo$4bobo13b2o6bo7b2o$4bobo12bo2bo2bo2bo7b2o$5bo13b3o6bo$17b3o5bobo$5b2o9bobo6b2o$5b2o9bo$15b2o!&name=(p30) glider loop}{\includegraphics[width=\textwidth]{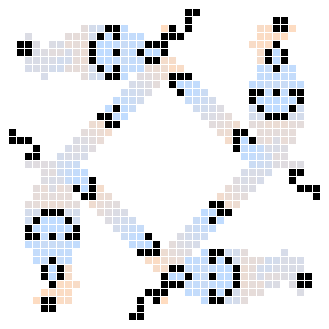}}}
  \hfill
  \patstack[0.25]{(p300) glider loop}{\href{https://conwaylife.com/?rle=37b2o$37bo9b2o$27b2o6bobo9b2o$26bobo5b3o$25bo6b3o13bo$16b2o7bo2bo2bo2bo12bobo$16b2o7bo6b2o13bobo$26bobo19bo$27b2o2$45b2obob2o$45bo5bo$46bo3bo$47b3o7$50b2o$50bo$51b3o$53bo7$o$3o$3bo$2b2o7$4b3o$3bo3bo$2bo5bo$2b2obob2o8bo$18bo$16b3o6b2o$5bo19bobo$4bobo13b2o6bo7b2o$4bobo12bo2bo2bo2bo7b2o$5bo13b3o6bo$17b3o5bobo$5b2o9bobo6b2o$5b2o9bo$15b2o!&name=(p300) glider loop}{\includegraphics[width=\textwidth]{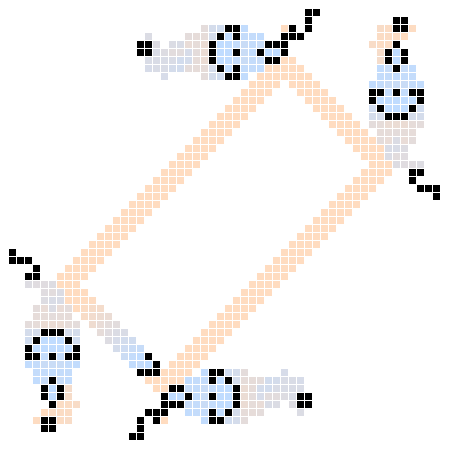}}}
\end{center}

Rather than using a perturbation in a still life as a signal in a loop, an obvious alternative is to use the glider, which propagates itself without any external intervention. All that is necessary is some method of reflecting a glider by 90 degrees. Above, we show some glider loops formed by four copies of the period-30 \emph{buckaroo}, a variant of the queen bee shuttle that produces a spark capable of reflecting a glider. On the left, we have the most basic version: a single glider is reflected four times to return to its starting position. We can produce oscillators with factors of its period by placing additional gliders in the loop (as long as the period is still divisible by 30, so that the gliders arrive at each buckaroo when it is in the correct phase). The pattern will repeat when each glider moves into the position and state of the next, rather than when one glider goes the entire way around. In the center pattern we fit 6 gliders into the same loop, producing the minimum period of 30. The positions of the reflectors are adjustable and by moving the buckaroos further apart, larger periods can be produced. Combining these two tricks, an oscillator with period $30n$ can be produced for any $n$. Glider loops can be constructed using many different glider reflectors, each giving an infinite family of oscillators.

\begin{center}
\patstack[0.5]{(p49) glider loop}{\href{https://conwaylife.com/?rle=14bo2bo2bo$14b7o2bo$21b3o$16b2obo$15bo3b3o$15bobo4bo$14b2o2bob2obo$15bobo3bobo$15bob2obo2b2o$13bobo4bobo$13b2obobobobo10bo$16bobobob2o7b3o$16bo4bo8bo$13b2ob2obobo8b2o$13bobo3bob2o$16b4o8bo$20b2o2b2obobo$16b4o8b2o2b2o$16bo2bob2o9b2o$21bo$19bobo8bo19b2o2b2o$19b2o8b2o20bo2bo3bo4b2o$29bobo18bo4b6o2bo$46b2obob4o6bobo$10b2o5b2o28bobobo4b2obobob2o$11bo5b2o28bobo3b2obobo4bo$11bobo29b2ob2ob3o8b2obo$12b2o29bo4bo4b4obobo2b2o$16b2o26b3obob3o4b2obobo$15bobo4bo23bo3bo2b3o3bo2bo$16bo6bo29bo2b3o3b2o$21b3o24bo7bo$16bo31bo$8bo7bo24b3o$b2o3b3o2bo29bo6bo$2bo2bo3b3o2bo3bo23bo4bobo$2bobob2o4b3obob3o26b2o$2o2bobob4o4bo4bo29b2o$bob2o8b3ob2ob2o29bobo$bo4bobob2o3bobo28b2o5bo$2obobob2o4bobobo28b2o5b2o$bobo6b4obob2o$bo2b6o4bo18bobo$2o4bo3bo2bo20b2o8b2o$9b2o2b2o19bo8bobo$43bo$31b2o9b2obo2bo$31b2o2b2o8b4o$35bobob2o2b2o$36bo8b4o$42b2obo3bobo$33b2o8bobob2ob2o$34bo8bo4bo$31b3o7b2obobobo$31bo10bobobobob2o$42bobo4bobo$40b2o2bob2obo$41bobo3bobo$41bob2obo2b2o$42bo4bobo$43b3o3bo$45bob2o$41b3o$41bo2b7o$44bo2bo2bo!&name=(p49) glider loop}{\includegraphics[width=\textwidth]{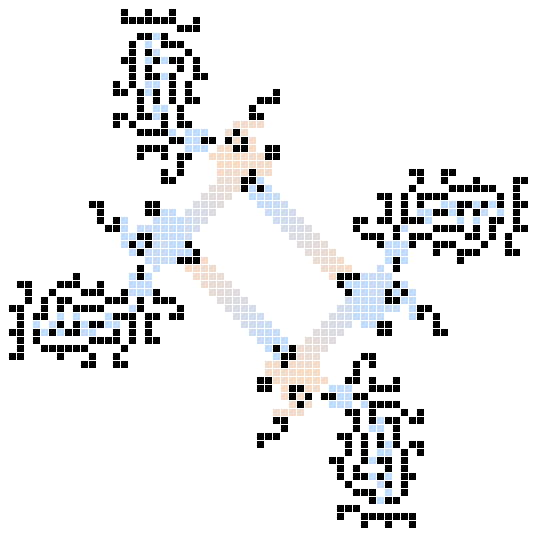}}}
\end{center}

A notable example is the \emph{bouncer}, a periodic glider reflector discovered in the late 90s by Elkies, based on the collision of a glider with a still life known as a \emph{boat}. A block, an eater, and a spark perturb the reaction such that it releases another glider and restores the boat. The spark interacts briefly and only once per glider reflection, so the reaction can work with a sparker of any sufficiently large period. This was used to form the first known oscillator of period 49, displayed above, using period-7 sparkers.

This particular loop can be modified to have any larger period that is a multiple of 7. But there is an obstruction to adjusting it much lower: if the period is too low, the reflection reaction may not have completed by the time the next glider arrives! A gap of 22 generations is sufficient for the reflections to succeed. For reflectors and other mechanisms, this is called the \emph{minimum repeat time}: the minimum delay between two inputs such that the mechanism works as expected both times. In general, for any period $p \geq 22$, a bouncer-type oscillator can be assembled as long as a suitable sparker is known. To function as a loop, the period of the sparker used must divide the overall length of the loop.

\subsubsection{Herschel Conduits}

\begin{center}
\patstack{The Herschel}{\href{https://conwaylife.com/?rle=o$3o$obo$2bo!&name=Herschel}{\includegraphics[width=\textwidth]{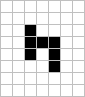}}}
\end{center}

Besides the glider, a promising choice of active region for forming signal loops is the \emph{Herschel}, an active pattern that often appears in the evolution of other more common sequences, such as the R-pentomino and B-heptomino. Crucially, the Herschel typically appears a fairly long distance from the starting point, providing lots of space for catalysts to tame the remainder of the reaction. 

\begin{center}
\patstack[0.6]{A length 724 Herschel loop}{\href{https://conwaylife.com/?rle=71b2o$71b2o5b2o$60b2o16b2o$42b2o17bo$43bo17bobo$30bo11bo19b2o12b2o$30b3o9b2o32b2o$33bo48b2o$20b2o10b2o48b2o$21bo$21bobo$7b2o13b2o50b2o$7b2o36b2o15bo$45b2o15bobo$62b3o$b2o61bo$b2o$5b2o$5b2o15bo$11bo10bobo39b2ob2o$11bo10b3o29bob2o4bo2bobobo16bo$24bo11b2o14b3ob2o4b2obo2bo15b3o$2o34bo14bo13bo8b3o6bo$2o35b3o12b3ob2o7b2o7bo8b2o$39bo14bobo6b2o2bobo3b3o$20bo3b2o28bobo5bo2bo2b2o$19bobo3bo29bo7b2o$19b2o3bo$23bo$19b5obo$13b3o3bo4bobo61b2o$4b2o8bo6bo2bobo61bo$5bo6b3o5b2o3bo60bobo$2b3o81b2o$2bo2$72b2o$20b2obo47bobo$20b2ob3o45bo$26bo43b2o$20b2ob3o$21bobo$21bobo64b2obo$22bo65bob2o2$81b2o$81b2o2$12b2o$12b2o2$3b2obo65bo$3bob2o64bobo$71bobo$69b3ob2o$23b2o43bo$23bo45b3ob2o$21bobo47bob2o$21b2o2$92bo$7b2o81b3o$6bobo60bo3b2o5b3o6bo$6bo61bobo2bo6bo8b2o$5b2o61bobo4bo3b3o$69bob5o$71bo$70bo3b2o$30b2o7bo29bo3bobo$25b2o2bo2bo5bobo28b2o3bo$19b3o3bobo2b2o6bobo14bo$10b2o8bo7b2o7b2ob3o12b3o35b2o$11bo6b3o8bo13bo14bo34b2o$8b3o15bo2bob2o4b2ob3o14b2o11bo$8bo16bobobo2bo4b2obo29b3o10bo$26b2ob2o39bobo10bo$72bo15b2o$88b2o$92b2o$30bo61b2o$30b3o$30bobo15b2o$32bo15b2o36b2o$19b2o50b2o13b2o$71bobo$73bo$11b2o48b2o10b2o$11b2o48bo$17b2o32b2o9b3o$17b2o12b2o19bo11bo$31bobo17bo$33bo17b2o$15b2o16b2o$15b2o5b2o$22b2o!&name=A length 724 Herschel loop}{\includegraphics[width=\textwidth]{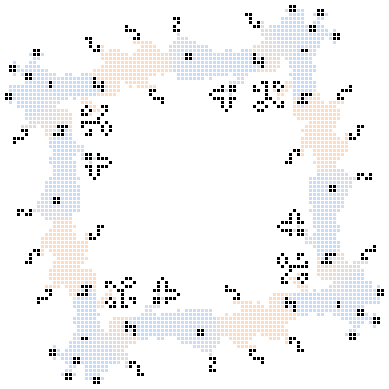}}}
\end{center}

Through the 1990s, many \emph{Herschel conduits} were found --- that is, placements of catalysts that cleanly restore a Herschel at a different location. In particular, Buckingham discovered~\cite{buckingham:conduits} a set of Herschel conduits that can be used to construct arbitrarily long loops. 

Above, we give an example of a Herschel loop built out of two elementary components: the `F117' conduit which translates a Herschel forwards in 117 generations, and the `R64' conduit, which translates and turns a Herschel to the right after 64 generations. A Herschel is placed at the input/output locations of each component; these are not equally spaced as the two types of component take different durations to traverse. We have used four of each component, so a single Herschel travels around the loop in $4 \times 117 + 4 \times 64 = 724$ generations. By deleting all but one Herschel, we have an oscillator with period equal to the full length of the track: 724. With two evenly spaced Herschels the period is 362, and with four, the period is 181. This is the lowest we can manage with this conduit: the next smallest factor of 724 is 181, and there is certainly no space to fit 181 Herschels. 

Crucially, the catalysts used in such conduits are all stable, so the length of the loop is not constrained by the use of periodic sparkers or reflectors. By constructing different loops and placing Herschels with regular spacings, oscillators of \emph{all} sufficiently large periods can be constructed. Buckingham's initial set of conduits found in 1996 established that all periods $p \geq 61$ could be achieved, the bound coming from the minimum repeat time of the components used. A careful argument that all periods $p \geq 61$ are attainable with Herschel loops is given in \cite[\S 3.6]{life:book}.

Some additional Herschel conduits with lower repeat times were discovered the following year, and this bound was compressed first to $p \geq 58$ by Buckingham and Paul Callahan~\cite{buckingham-callahan:bounds}, and later to $p \geq 56$ by Dietrich Leithner using sparkers~\cite{leithner:fast}.


\subsubsection{Stable Reflectors and the Snark}

Herschel conduits are limited by their relatively long minimum repeat time and glider loops are limited by the period of the reflector used. There was a long held hope that a \emph{stable} glider reflector could be found that would push the bound of unknown periods lower still.

In the earliest days of Life, it was known that a stable glider reflector must exist. This follows as a theoretical consequence of two important ideas. First, it was argued that \emph{universal computers} exist~\cite[\S 25]{winning-ways:vol2} \cite{wainwright:universal, the-recursive-universe}. That is, one can construct a pattern to simulate an arbitrary Turing machine. It was also argued that \emph{universal constructors} exist: patterns that can follow instructions to construct any \emph{glider-synthesisable} object, that is, any object that can be created by colliding gliders in some combination. In particular, when fed the correct instructions, such a constructor is capable of building both a distant copy of itself, and a distant copy of a universal computer. 

Together, this implies the existence of a stable reflector; the argument goes as follows~\cite[\S 4]{buckingham-callahan:bounds}. The incoming glider hits a cluster of still lifes, destroying them and spawning additional gliders. These gliders can be redirected through additional destructive collisions, and with sufficient time, an almost arbitrary flotilla of gliders can be produced. The flotilla chosen is one that synthesises a universal computer/constructor combination as described above, together with a program that instructs the computer to: 1) Emit a glider at 90 degrees from the input; 2) Reconstruct the initial configuration of still lifes; and, 3) Cleanly self-destruct the universal computer/constructor. This program can be encoded in the position of a single distant still life to be read by the computer, to avoid any concerns about the complexity of having the computer reconstruct a description of itself.

Given a stable glider reflector, we can build a loop from 8 of these reflectors\footnote{We use 8 reflectors rather than 4 to avoid any complications that would arise if the delay imposed on the glider by the reflector were a multiple of 4.} whose length can be any sufficiently large multiple of 8. Placing gliders in the loop at regular intervals, we can achieve any sufficiently large period. This method is massively inefficient: we will see next that very compact stable glider reflectors exist. Still, it was known from the earliest days that at most finitely many oscillator periods could fail to exist; all that remains is stamp collecting.

These early arguments were not constructive, and it was only much later that explicit Turing machines~\cite{rendell:turing-machine, apg:universal} and universal \emph{construction arms}~\cite[\S 11.1]{life:book} \cite{ekstrom:new-arms} were created. An example of a large, universal constructor-based oscillator has been constructed by forum user Goldtiger997~\cite{goldtiger:rro}.

The first explicit stable glider reflector was constructed by Callahan~\cite{reflector-history}, where the chaos from the collision of a glider with a bait object is turned into a Herschel, and Herschel conduits are used to emit another glider and restore the original collision site. The repeat time was quickly reduced from 4840 down to 850 generations, and conduit-based constructions by others eventually reached 444. This was still far too high to construct oscillators with new periods.

\begin{center}
\patstack[0.2]{The Snark}{\href{https://conwaylife.com/?rle=2bo$obo$b2o2$17bo$15b3o$14bo$14b2o5b3o$23bo$22bo2$11bobo$12b2o$12bo$4b2o$3bobo5b2o$3bo7b2o$2b2o2$16bo$12b2obobo$11bobobobo$8bo2bobobobob2o$8b4ob2o2bo2bo$12bo4b2o$10bobo$10b2o!&name=The Snark}{\includegraphics[width=\textwidth]{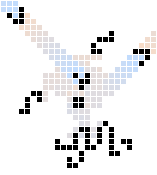}}}
\end{center}

In 2013, Mike Playle used his Bellman~\cite{bellman, bellman-manual} catalyst search program to discover the \emph{Snark}, the first small, stable 90-degree glider reflector. The glider hits a block creating a symmetrical active region known as a \emph{honey farm}, which is shaped by three nearby stable catalysts so as to produce an output glider perpendicular to the input and restore the block, without going through a Herschel as an intermediate stage. 

\begin{center}
\patstack[0.4]{(p43) Snark loop}{\href{https://conwaylife.com/?rle=36b2o$35bobo$29b2o4bo$27bo2bo2b2ob4o$27b2obobobobo2bo$30bobobobo$30bobob2o$31bo2$44b2o$35b2o7bo$35b2o5bobo$42b2o$35bo$34b2o$34bobo2$25bo$24bo$9bo14b3o5b2o$9b3o21bo$12bo17b3o$11b2o17bo2$45b2o$3b2o40bobo$3bo41bo$2obo56b2o$o2b3o4b2o3bo44bo$b2o3bo3b2ob2o47bo$3b4o7b2o26b2o14b5o$3bo15b2o22bo13bo$4b3o12bobo21bobo12b3o$7bo13bo22b2o15bo$2b5o14b2o26b2o7b4o$2bo47b2ob2o3bo3b2o$4bo44bo3b2o4b3o2bo$3b2o56bob2o$19bo41bo$17bobo40b2o$18b2o2$34bo17b2o$32b3o17bo$31bo21b3o$31b2o5b3o14bo$40bo$39bo2$28bobo$29b2o$29bo$21b2o$20bobo5b2o$20bo7b2o$19b2o2$33bo$29b2obobo$28bobobobo$25bo2bobobobob2o$25b4ob2o2bo2bo$29bo4b2o$27bobo$27b2o!&name=(p43) Snark loop}{\includegraphics[width=\textwidth]{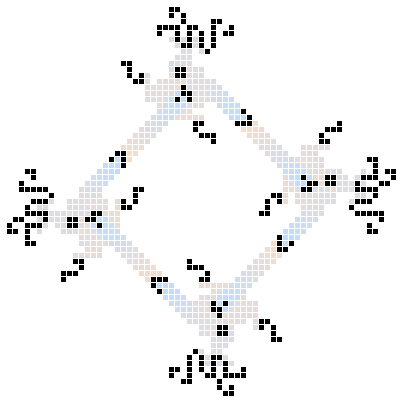}}}
\end{center}

The Snark's minimum repeat time is 43 generations. As it is a stable reflector, it can also reflect at any larger spacing, giving a simple and uniform construction for oscillators of all periods $p \geq 43$. Glider loops using the Snark provided the first known oscillators with periods 43 and 53~\cite{merzenich-snark}. So that our presentation of omniperiodicity is completely self-contained, we give a proof that all periods $p \geq 43$ are indeed attainable.

\begin{theorem*}
For any $p \geq 43$, a Snark loop with period $p$ can be constructed.
\end{theorem*}
\begin{proof}
Place four Snarks in a rectangular configuration, as in the case of the period-43 oscillator above.

This produces a glider path in the shape of a diagonal rectangle, with a width of $n$ diagonals and a height of $m$ diagonals. The spacings $n$ and $m$ must be large enough so that the Snarks do not destabilize each other ($n, m > 13$ is sufficient). The time taken for one glider to traverse this loop is given by the perimeter of the loop, $2n+2m$, divided by the speed of the glider, $1/4$, plus four times the 2-generation delay imposed by each Snark reflection. In all, this is $4 \times (2n+2m) + 8 = 8(n+m+1)$.

Inserting 8 equally spaced gliders into this loop will create an oscillator with a period $p = n+m+1$, so long as $p \geq 43$. Thus, for any $p \geq 43$ there is a Snark loop with period $p$ taking, for concreteness, $n=\lfloor (p-1)/2 \rfloor$ and $m=\lceil (p-1)/2 \rceil$.
\end{proof}


\subsection{Catalyst Searches}

\begin{center}
\hfill
\patstack[0.3]{(p19) cribbage}{\href{https://conwaylife.com/?rle=4b2o$4bo$b2obo10bo$bo2b2o9b3o$3bo2bo11bo$bob4o10b2o$obo$o2b4o12b3o$b2o3bo11bo3bo6b2o$3b2o5b3o4bo5bo4bobo$3bo5bo3bo4bo3bo5bo$bobo4bo5bo4b3o5b2o$b2o6bo3bo11bo3b2o$10b3o12b4o2bo$29bobo$13b2o10b4obo$13bo11bo2bo$14b3o9b2o2bo$16bo10bob2o$27bo$26b2o!&name=(p19) cribbage}{\includegraphics[width=\textwidth]{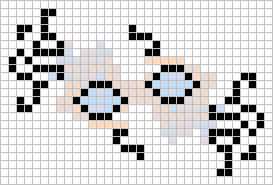}}}
\hfill
\patstack[0.2]{(p34) 74P34}{\href{https://conwaylife.com/?rle=15bo$13b3o$3b2o7bo$4bo7b2o$4bobo$5b2o2$2o$bo$bobo$2b2o11bo$15bo9b2o$9b3o4bo8bo$9bobo11bobo$9bobo11b2o2$2b2o11bobo$bobo11bobo$bo8bo4b3o$2o9bo$11bo11b2o$23bobo$25bo$25b2o2$20b2o$20bobo$13b2o7bo$14bo7b2o$11b3o$11bo!&name=(p34) 74P34}{\includegraphics[width=\textwidth]{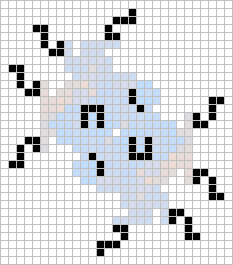}}}
\hfill
~
\end{center}

Just as \texttt{dr} can be seen as a systematic search for billiard tables, we can also attempt a systematic search for hasslers. Hasslers have three components: the active pattern, the catalysts, and the catalyst placements. The active regions that have the best chance of forming hasslers are the ones that occur most often naturally: after being perturbed, they have the best chance of reforming in the correct position. Simon Ekström \cite{small-census} performed a census of 10 million small active patterns, and the resulting list of the most common patterns forms the basis of many of the following searches.

Similarly, the most useful catalysts are the ones that have the best chance of perturbing an active pattern and surviving. These evaluations were run by Dongook Lee~\cite{catalyst-test}, and with a more expansive list by Mitchell Riley~\cite{catalyst-test-mvr}, resulting in a shortlist of the catalysts with the most potential to produce hasslers.

Finally, these catalysts need to be placed in the hopes that the active region will reform. Like drifter searches, each result of a catalyst-based search is a roll of the dice: there is no way of specifying in advance which period you would like to find.

One approach is to simply brute-force every possible position of the catalysts: this is the approach taken by Michael Simkin's \texttt{CatForce}~\cite{catforce}. Such brute-force searches are only feasible for small numbers of catalysts chosen from a small number of options. Instead, we can place catalysts ahead of the reaction ``just in time'' to perturb it, and bail out of testing a configuration if the catalysis fails or one of the catalysts is later destroyed by the active pattern. This can naturally be written as a depth-first search through catalyst positions, and is the approach taken by the \texttt{ptbsearch}~\cite{ptbsearch} and \texttt{Catalyst}~\cite{catalyst} programs.

A second improvement is searching specifically for symmetrical hasslers. Allowing symmetry greatly enlarges the available search space, more than simply searching for hasslers with an additional catalyst would, as there are far more ways for copies of the active region to interact with one another than there are ways to interact with an additional catalyst successfully. 

A search script by Raucci~\cite{hdp-spark} incorporates both of these improvements, focusing on oscillators supported by sparks. Rather than placing entire catalysts, an isolated spark is placed to interact with the active region. If the active region then reoccurs, the ``partial'' oscillator can sometimes be completed to a true oscillator by placing sparkers of an appropriate period. 

\texttt{Symmetric CatForce}~\cite{symmetric-catforce} is a modified version of (asymmetric) \texttt{CatForce} that also incorporates these ideas, using a fast Life-step algorithm \cite{rokicki:algorithms} by Tomas Rokicki. The first known oscillators of periods 34 and 19 were found by Riley using \texttt{Symmetric CatForce}, both coincidentally having 2-fold rotational symmetry. 

\subsubsection{Dependent Reflectors}

\begin{center}
\patstack[0.35]{(p31) dependent glider loop}{\href{https://conwaylife.com/?rle=35b2o$35b2o$28bo$27bobo$26bo3bo$2o8bo15bo3bo$2o6b3o15bo3bo$7bo7bo11bobo3b2o$7b2o6b2o11bo4bobo$14bobo18bo$4b3o28b2o$3bo3bo$2bo5bo$2bo5bo$2bo5bo22bobo$3bo3bo24b2o$4b3o25bo8$8bo25b3o$7b2o24bo3bo$7bobo22bo5bo$32bo5bo$32bo5bo$33bo3bo$4b2o28b3o$5bo18bobo$5bobo4bo11b2o6b2o$6b2o3bobo11bo7bo$10bo3bo15b3o6b2o$10bo3bo15bo8b2o$10bo3bo$11bobo$12bo$4b2o$4b2o!&name=(p31) dependent glider loop}{\includegraphics[width=\textwidth]{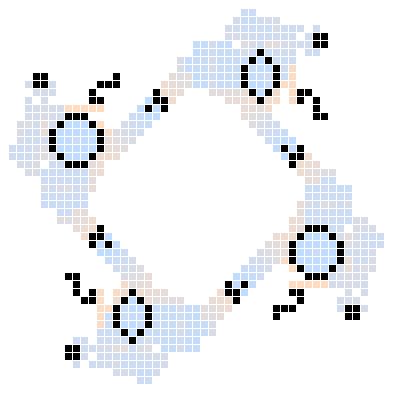}}}
\end{center}

Among hasslers, a large, fruitful, and tractable search space exists with \emph{dependent glider reflectors}. These reflectors are much like traditional glider reflectors such as the Snark, but are dependent on a regularly-spaced stream of input gliders. Any interruption or disturbance in the glider stream destroys the reflector. There is no assumption that a glider will strike a bait object to initiate the reflection sequence. Instead, the glider will often strike an actively-evolving pattern in the dependent reflector, and this disturbed pattern will be guided back to its original location while emitting another glider. 

This is demonstrated above in a period-31 dependent glider loop discovered by Merzenich. The honey farm interacts with the eater in precisely the same way as used in the Snark reflector, emitting a glider. Rather than simply settling down to a block, the active region interacts with the incoming glider and outermost block to reform a honey farm directly.

Dependent reflectors are far easier to find than stable reflectors. Since evolving active regions both 1) are larger than the still lifes useful in traditional glider reflectors and 2) have more than one phase at which the glider can strike them, there are many more available possibilities for the initial glider collision. As an extreme example, a glider striking a block has only three distinct reactions that could lead to a stable reflector --- one yielding a pi-heptomino, one a honey farm, and one that annihilates the block and glider in 28 generations. Meanwhile, dependent reflector targets can easily exceed 100 distinct useful interactions. 

A search for dependent reflectors can be seen as a hassler search where the incoming glider is treated as a ``catalyst'' that does not have to recover. In this way, dependent reflectors enlarge the available search space in much the same way as searching with symmetry. Of course, to actually be a ``reflector'', the reaction has to at some point emit another glider.



\begin{center}
\patstack[0.5]{(p41) 204P41}{\href{https://conwaylife.com/?rle=34bo$32b3o$31bo$16b2o13b2o$17bo$17bobo$18b2o$29bo$28bobo$27bo3bo$27bo3bo$27bo3bo$28bobo$29bo2$5b2o24b2o$6bo24bo8bobo$4bo14b2o11b3o6b2o$2b4o10b2o2bo13bo6bo$bo12bob2o3b4o$o2b3o8b2o3bo3b2obo$b2o2bo8bob2o3bo3b2o$3b2o5b2o4b4o3b2obo$3bo5b2o9bo2b2o$4bo6bo8b2o$b3o$bo5b2o15bo27bo$8bo14bobo24bobo$5b3o15b2o26b2o$5bo2$67bo$20b2o26b2o15b3o$20bobo24bobo14bo$20bo27bo15b2o5bo$69b3o$51b2o8bo6bo$48b2o2bo9b2o5bo$46bob2o3b4o4b2o5b2o$46b2o3bo3b2obo8bo2b2o$46bob2o3bo3b2o8b3o2bo$48b4o3b2obo12bo$31bo6bo13bo2b2o10b4o$30b2o6b3o11b2o14bo$30bobo8bo24bo$40b2o24b2o2$43bo$42bobo$41bo3bo$41bo3bo$41bo3bo$42bobo$43bo$53b2o$53bobo$55bo$40b2o13b2o$41bo$38b3o$38bo!&name=(p41) 204P41}{\includegraphics[width=\textwidth]{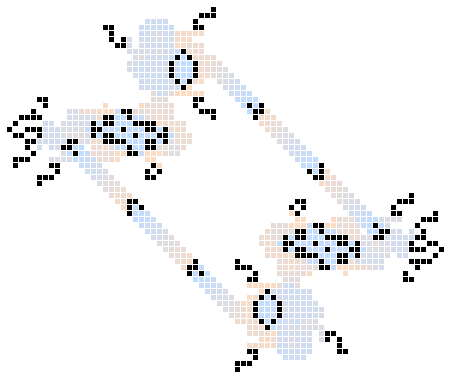}}}
\end{center}

\texttt{LocalForce}~\cite{localforce} by Nico Brown is a variant of \texttt{CatForce} that supports placing gliders as catalysts, and it quickly yielded numerous dependent reflectors based on the honey-farm reaction. 
Searches with other active regions yielded numerous partial results where the active region repeated with period 41, but did not release the glider required to act as a dependent reflector. While these partials did not create gliders, several created sparks. Upon further modifying \texttt{LocalForce} to search for interactions involving these sparks, a period-41 dependent reflector combining two active regions was quickly found, yielding a period-41 loop and proving that Life is omniperiodic.

\subsection{Future}

The omniperiodicity problem has served as a rallying point for the Life community, leading to the creation of the many tools and programs we have detailed above. 

However, there are still vast gaps in our knowledge about what is possible in Life~\cite{lifewiki:status-page}. Perhaps the most significant is which of the theoretically possible\footnote{No pattern can propagate faster than the `speed of light' $c$, the distance of one cell per generation. For spaceships moving through empty space, the limits are much lower: $c/2$ for orthogonal motion and $c/4$ for diagonal motion~\cite[\S 4.5]{life:book}.} \emph{spaceship} velocities can be realised. Recall that a spaceship is a periodic pattern that reoccurs at some translation from its original position; one might think of oscillators as stationary spaceships. Similarly to oscillators, there are non-constructive arguments that show that all sufficiently slow spaceship velocities are attainable, but low-period spaceships with new velocities are exceptionally difficult to find. A notable example is Sir Robin~\cite{apg:sir-robin}, the first discovered elementary spaceship with a non-orthogonal and non-diagonal velocity. 

Another class of periodic objects is \emph{glider guns};  patterns that release gliders at regular intervals. Glider guns are known that produce glider streams of all possible\footnote{Glider guns must have a period $p \geq 14$ for the gliders to not interfere with each other.} output periods. However, for a handful of periods, no ``true-period" gun exists: the gun does not have the same period as the output stream. For example, the known gun that produces a period-15 glider stream consists of the period-30 Gosper glider gun together with a second period-30 mechanism that inserts a glider into each gap. At the time of writing, the remaining true-period guns to be found have period $14 \leq p \leq 19$, and $p = 23, 26, 29, 31, 35, 38, 39, 47$ and $53$. 

\begin{center}
\hfill
\patstack[0.2]{(p28) newshuttle}{\href{https://conwaylife.com/?rle=26b2o$20bo3bo2bo2bo$18b3o3b3o3b3o$8b2o7bo15bo7b2o$9bo7b2o5b3o5b2o7bo$9bobo11bo3bo11bobo$10b2o11b2ob2o11b2o2$3bo43bo$3b3o39b3o$6bo7b3o17b3o7bo$5b2o7bobo17bobo7b2o$14b3o3b3o5b3o3b3o$20bobo5bobo$10b3o7b3o5b3o7b3o$10bobo25bobo$10b3o25b3o$3b2o18b2ob2o18b2o$2bobo18bo3bo18bobo$2bo21b3o21bo$b2o9b3o11bobo7b3o9b2o$12bobo12b2o7bobo$12b3o5b2o14b3o$2o3b2o10b2o2bo10b2o10b2o$obobobo10bob2o10bobo10bobob2o$2bobo14bo11bo14bobo$b2obobo10bobo10b2obo10bobobobo$5b2o10b2o10bo2b2o10b2o3b2o$12b3o14b2o5b3o$12bobo7b2o12bobo$b2o9b3o7bobo11b3o9b2o$2bo21b3o21bo$2bobo18bo3bo18bobo$3b2o18b2ob2o18b2o$10b3o25b3o$10bobo25bobo$10b3o7b3o5b3o7b3o$20bobo5bobo$14b3o3b3o5b3o3b3o$5b2o7bobo17bobo7b2o$6bo7b3o17b3o7bo$3b3o39b3o$3bo43bo2$10b2o11b2ob2o11b2o$9bobo11bo3bo11bobo$9bo7b2o5b3o5b2o7bo$8b2o7bo15bo7b2o$18b3o3b3o3b3o$20bo2bo2bo3bo$23b2o!&name=(p28) newshuttle}{\includegraphics[width=\textwidth]{p28}}}
\hfill
\patstack[0.2]{(p28) lcm(4, 14) oscillator}{\href{https://conwaylife.com/?rle=9b2o5b2o$9b2o5b2o4$9b2o5b2o$8bo2bo3bo2bo$7bo3bo3bo3bo$11b2ob2o$b2o3bo5bobo5bo$o2bo3bo3bo3bo3bo$obo2bo5bo3bo$bo6bo9bo$2b2obo2bo2bo3bo2bo$4bo4b2o5b2o!&name=lcm(4, 14) oscillator}{\includegraphics[angle=90, width=\textwidth]{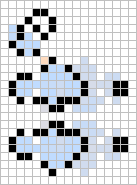}}}
\hfill
~
\end{center}

Even without leaving the search for oscillators, there is still much to do. For example, there is interest in finding oscillators with as low a population as possible, so-called Smallest Known Oscillators of Period $p$~\cite{skopje, lifewiki:oscillator}. In the above sections we have generally shown the first known oscillator of each period, but most are no longer the smallest. For example, the period-28 newshuttle is over 6 times as large as a $\lcm(4, 14)$ oscillator, the smallest known period-28 oscillator at the time of writing.

The oscillator hunt continues if we add criteria that the oscillators must satisfy. Indeed, the rules have already changed at least once as, in the past, LCM oscillators were considered ``too boring to count'' as filling a period. This is just a matter of aesthetics, but those worried about the interestingness of oscillators can still be satisfied: at least one non-LCM oscillator exists for each period first found using an LCM oscillator, so what one might call \emph{omni-interesting-periodicity} holds.

\begin{center}
\hfill
\patstack[0.2]{(p2) phoenix}{\href{https://conwaylife.com/?rle=3bo$3bobo$bo$6b2o$2o$6bo$2bobo$4bo!&name=(p2) phoenix}{\includegraphics[width=\textwidth]{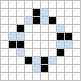}}}
\hfill
\patstack[0.23]{statorless p5}{\href{https://conwaylife.com/?rle=bo11bo$2o11b2o$o13bo$2o2b3ob3o2b2o$2ob2o5b2ob2o$bo2b2o3b2o2bo$bo3bo3bo3bo$bo3bo3bo3bo$bo2b2o3b2o2bo$2ob2o5b2ob2o$2o2b3ob3o2b2o$o13bo$2o11b2o$bo11bo!&name=statorless p5}{\includegraphics[width=\textwidth]{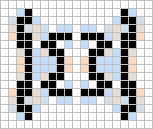}}}
\hfill
~
\end{center}

A class of oscillators for which we do not have all periods is \emph{strictly volatile} oscillators: patterns where every cell oscillates at the full period. The vast majority of oscillators we have seen fail this requirement: of all the oscillators displayed in this document, only figure eight, tumbler, pentadecathlon and Rob's p16 are strictly volatile. Brown~\cite{brown:strictly-volatile} gives a systematic construction for strictly volatile oscillators with all periods $p \geq 949$, using opposing flotillae of spaceships to construct each other. Below that, very little is known; even period-7 is unresolved. Above we have displayed a strictly volatile period-2 oscillator known to the MIT group in the early 70s, and the only known strictly volatile period-5 oscillator, discovered by Josh Ball in 2016. Most search techniques we have seen are useless here. For example, catalyst searches and Herschel conduits rely on stable catalysts, and drifter searches typically result in large stable supports.

The use of constraint solvers is an under-explored search technique. The existence of an oscillator of some size and period can be expressed as a Boolean satisfiability problem and passed to a standard SAT solver~\cite[p.~201--204]{taocp:4b} \cite{lls}. The periods for which this is currently effective are still fairly low, however. Constraint solvers have helped resolve other questions about Life, such as the ``generalized grandfather problem''~\cite{salo-torma:oracles}, the maximum-density still life problem~\cite{chu-stuckey:maximum-density}, and the computational universality of determining whether a pattern can be reversed even a single step \cite{salo-torma:backwards}.

\medskip
Life's work is never done!

\newpage

\section*{\hfil First Known Oscillators of All Periods \hfil}\label{sec:table}

\begin{center}
\begin{longtable}{llll}
Period & Name & Discoverer & Year \\ \midrule \endhead
1 & \href{https://conwaylife.com/wiki/Block}{block} & John H. Conway & 1969 \\
2 & \href{https://conwaylife.com/wiki/Blinker}{blinker} & John H. Conway & 1969 \\
3 & \href{https://conwaylife.com/wiki/Pulsar}{pulsar} & John H. Conway & 1970 \\
4 & \href{https://conwaylife.com/wiki/Pinwheel}{pinwheel} & Simon Norton & 1970 \\
5 & \href{https://conwaylife.com/wiki/Octagon_2}{octagon 2} & Sol Goodman, Arthur C. Taber & 1971 \\
6 & \href{https://conwaylife.com/wiki/$rats}{\$rats} & David Buckingham & 1972 \\
7 & \href{https://conwaylife.com/wiki/Burloaferimeter}{burloaferimeter} & David Buckingham & 1972 \\
8 & \href{https://conwaylife.com/wiki/Figure_eight}{figure eight} & Simon Norton & 1970 \\
9 & \href{https://conwaylife.com/wiki/Worker_bee}{worker bee} & David Buckingham & 1972 \\
10 & \href{https://conwaylife.com/wiki/42P10.3}{42P10.3} & David Buckingham & $\leq$1976 \\
11 & \href{https://conwaylife.com/wiki/38P11.1}{38P11.1} & David Buckingham & 1977 \\
12 & \href{https://conwaylife.com/wiki/Dinner_table}{dinner table} & Robert T. Wainwright & 1972 \\
13 & \href{https://conwaylife.com/wiki/Buckingham%27s_p13}{Buckingham's p13} & David Buckingham & 1976 \\
14 & \href{https://conwaylife.com/wiki/Tumbler}{tumbler} & George D. Collins, Jr. & 1970 \\
15 & \href{https://conwaylife.com/wiki/Pentadecathlon}{pentadecathlon} & John H. Conway & 1970 \\
16 & \href{https://conwaylife.com/wiki/Two_pre-L_hasslers}{two pre-L hassler} & Robert T. Wainwright & 1983 \\
17 & \href{https://conwaylife.com/wiki/54P17.1}{54P17.1} & Dean Hickerson & 1997 \\
18 & \href{https://conwaylife.com/wiki/117P18}{117P18} & David Buckingham & $\leq$1991 \\
19 & \href{https://conwaylife.com/wiki/Cribbage}{cribbage} & Mitchell Riley & 2023 \\
20 & $\lcm(4,5)$ oscillator & ? & $\leq$1992 \\
21 & $\lcm(3,7)$ oscillator & ? & $\leq$1992 \\
22 & \href{https://conwaylife.com/wiki/168P22.1}{168P22.1} & Noam D. Elkies & 1997 \\
23 & \href{https://conwaylife.com/wiki/David_Hilbert}{David Hilbert} & Luka Okanishi, praosylen & 2019 \\
24 & $\lcm(3,8)$ oscillator & ? & $\leq$1992 \\
25 & \href{https://conwaylife.com/wiki/134P25}{134P25} & Noam D. Elkies & 1994 \\
26 & \href{https://conwaylife.com/wiki/P26_pre-pulsar_shuttle}{p26 pre-pulsar shuttle} & David Buckingham & 1983 \\
27 & \href{https://conwaylife.com/wiki/123P27.1}{123P27.1} & Noam D. Elkies & 2002 \\
28 & \href{https://conwaylife.com/wiki/Newshuttle}{newshuttle} & David Buckingham & 1973 \\
29 & \href{https://conwaylife.com/wiki/P29_pre-pulsar_shuttle}{p29 pre-pulsar shuttle} & David Buckingham & 1980 \\
30 & \href{https://conwaylife.com/wiki/Queen_bee_shuttle}{queen bee shuttle} & Bill Gosper, Robert W. April & 1970 \\
31 & \href{https://conwaylife.com/wiki/Merzenich%27s_p31}{Merzenich's p31} & Matthias Merzenich & 2010 \\
32 & \href{https://conwaylife.com/wiki/Gourmet}{gourmet} & David Buckingham & 1978 \\
33 & \href{https://conwaylife.com/wiki/258P3_on_Achim%27s_p11}{258P3 on Achim's p11} & Noam D. Elkies, Achim Flammenkamp & 1997 \\
34 & \href{https://conwaylife.com/wiki/74P34}{74P34} & Mitchell Riley & 2022 \\
35 & toaster on 44P7.2 & Dean Hickerson, David Buckingham & $\leq$1992 \\
36 & \href{https://conwaylife.com/wiki/P36_toad_hassler}{p36 toad hassler} & Robert T. Wainwright & 1984 \\
37 & \href{https://conwaylife.com/wiki/Beluchenko%27s_p37}{Beluchenko's p37} & Nicolay Beluchenko & 2009 \\
38 & \href{https://conwaylife.com/wiki/44P38}{Raucci's p38} & David Raucci & 2022 \\
39 & \href{https://conwaylife.com/wiki/134P39.1}{134P39.1} & Noam D. Elkies, David Buckingham & 2000 \\
40 & \href{https://conwaylife.com/wiki/P40_B-heptomino_shuttle}{p40 B-heptomino shuttle} & David Buckingham & $\leq$1991 \\
41 & \href{https://conwaylife.com/wiki/204P41}{204P41} & Nico Brown & 2023 \\
42 & unix on 44P7.2 & ? & $\leq$1992 \\
$\geq$43 & \href{https://conwaylife.com/wiki/P43_Snark_loop}{Snark loop} & Mike Playle & 2013 \\
44 & \href{https://conwaylife.com/wiki/P44_pi-heptomino_hassler}{p44 pi-heptomino hassler} & David Buckingham & 1992 \\
45 & \href{https://conwaylife.com/wiki/Pentadecathlon_on_snacker}{pentadecathlon on snacker} & Mark Niemiec, John H. Conway & $\leq$1992 \\
46 & \href{https://conwaylife.com/wiki/Twin_bees_shuttle}{twin bees shuttle} & Bill Gosper & 1971 \\
47 & \href{https://conwaylife.com/wiki/P47_pre-pulsar_shuttle}{p47 pre-pulsar shuttle} & David Buckingham & 1982 \\
48 & $\lcm(3, 16)$ oscillator & Bill Gosper, Achim Flammenkamp & 1994 \\
49 & \href{https://conwaylife.com/wiki/P49_glider_shuttle}{p49 glider loop} & Noam D. Elkies & 1999 \\
$50+40n$ & \href{https://conwaylife.com/wiki/P50_glider_shuttle}{p$(
50 + 40n)$ glider shuttle} & Dean Hickerson & 1992 \\
51 & \href{https://conwaylife.com/wiki/Beluchenko%27s_p51}{Beluchenko's p51} & Nicolay Beluchenko & 2009 \\
52 & \href{https://conwaylife.com/wiki/Four_eaters_hassling_lumps_of_muck}{lumps of muck hassler} & David Buckingham & 1977 \\
54 & \href{https://conwaylife.com/wiki/P54_shuttle}{p54 shuttle} & David Buckingham & $\leq$1976 \\
55 & \href{https://conwaylife.com/wiki/P55_pre-pulsar_hassler}{p55 pre-pulsar hassler} & David Buckingham & 1986 \\
56 & \href{https://conwaylife.com/wiki/P56_B-heptomino_shuttle}{p56 B-heptomino shuttle} & David Buckingham & $\leq$1991 \\
57 & \href{https://conwaylife.com/wiki/P57_Herschel_loop_1}{p57 Herschel loop} & Dietrich Leithner & 1997 \\
58 & \href{https://conwaylife.com/wiki/P58_toadsucker}{p58 toadsucker} & Bill Gosper, Mark Niemiec & 1994 \\
59 & \href{https://conwaylife.com/wiki/P59_Herschel_loop_1}{p59 Herschel loop} & David Buckingham & 1997 \\
$60 + 120n$ & \href{https://conwaylife.com/wiki/P60_glider_shuttle}{p60 glider shuttle} & ? & $\leq$1971 \\
$\geq$61 & \href{https://conwaylife.com/wiki/P61_Herschel_loop_1}{Herschel loops} & David Buckingham & 1996 \\
$66+24n$ & \href{https://conwaylife.com/wiki/P42_glider_shuttle}{p$(66+24n)$ glider shuttle} & Robert T. Wainwright & 1984 \\
72 & \href{https://conwaylife.com/wiki/Two_blockers_hassling_R-pentomino}{R-pentomino hassler} & Robert T. Wainwright & 1990 \\
$75+120n$ & \href{https://conwaylife.com/wiki/6_bits}{6 bits} & Robert T. Wainwright & 1984 \\
80 & $\lcm(5, 16)$ oscillator & Dean Hickerson, Bill Gosper & 1994 \\
84 & \href{https://conwaylife.com/wiki/Pi_orbital}{pi orbital} & Noam Elkies & 1995 \\
87 & \href{}{$\lcm(3,29)$ oscillator} & Bill Gosper & 1994 \\
$44n$ & \href{}{p$44n$ glider loop} & David Buckingham & 1992 \\
$30n$ & \href{}{p$30n$ glider loop} & ? & $\leq$1973 \\
$46n$ & \href{}{p$46n$ glider loop} & Bill Gosper & 1971 \\
$94n$ & p$94n$ glider loop & Dean Hickerson & 1990 \\
96 & \href{https://conwaylife.com/wiki/P96_Hans_Leo_hassler}{p96 Hans Leo hassler} & Noam D. Elkies & 1995 \\
100 & \href{https://conwaylife.com/wiki/Centinal}{centinal} & Bill Gosper & $\leq$1987 \\
102 & \href{}{p$(102+18n)$ traffic jam} & Dean Hickerson & 1994 \\
$100+4n$ & \href{}{p$(100+4n)$ traffic jam} & Bill Gosper & 1994 \\
105 & \href{}{pentadecathlon on 44P7.2} & ? & $\leq$1992 \\
108 & \href{https://conwaylife.com/wiki/P108_toad_hassler}{p108 toad hassler} & David Buckingham & $\leq$1991 \\
$100+10n$ & \href{}{p$(100+10n)$ traffic jam} & Bill Gosper & 1994 \\
116 & \href{}{$\lcm(4,29)$ oscillator} & Bill Gosper & 1994 \\
124 & \href{https://conwaylife.com/wiki/P124_lumps_of_muck_hassler}{p124 lumps of muck hassler} & Dean Hickerson & 1994 \\
125 & \href{}{p$25n$ glider loop} & Noam D. Elkies & 1996 \\
126 & \href{}{p$(102+6n)$ traffic jam} & Noam D. Elkies & $\leq$1996 \\
128 & \href{https://conwaylife.com/wiki/Period-256_glider_gun#Gallery}{p128 Herschel loop} & David Buckingham & 1991 \\
$135+120n$ & \href{https://conwaylife.com/wiki/106P135}{p$(135+120n)$ glider shuttle} & Bill Gosper & 1989 \\
$136+8n$ & \href{}{p$(136+8n)$ Herschel loop} & David Buckingham & 1991 \\
$144+72n$ & \href{https://conwaylife.com/wiki/Gunstar}{p$(144+72n)$ gunstar} & David Buckingham & 1990 \\
145 & \href{}{$\lcm(5,29)$ oscillator} & Bill Gosper & 1994 \\
$150+3n$ & \href{}{p$(150+3n)$ Herschel loop} & David Buckingham & 1991 \\
$150+5n$ & \href{}{p$(150+5n)$ Herschel loop} & Dean Hickerson, David Buckingham & 1992 \\
$152+4n$ & \href{}{p$(152+4n)$ Herschel loop} & David Buckingham & 1991 \\
$165+120n$ & \href{https://conwaylife.com/wiki/P165_glider_shuttle}{p$(165+120n)$ glider shuttle} & Robert T. Wainwright & $\leq$1989 \\
$100n$ & \href{}{p$100n$ glider loop} & Bill Gosper & $\leq$1987 \\
$230+40n$ & \href{https://conwaylife.com/wiki/P50_glider_shuttle#Gallery}{p$(230+40n)$ glider shuttle} & Dean Hickerson & 1992 \\
$240n$ & \href{}{p$240n$ glider loop} & John H. Conway & 1970 \\
246 & \href{https://conwaylife.com/wiki/P42_glider_shuttle}{pinball} variant & Dean Hickerson & 1992 \\
256 & \href{https://conwaylife.com/wiki/Period-256_glider_gun#Gallery}{p256 Herschel loop} & David Buckingham & 1991 \\
$270+24n$ & \href{https://conwaylife.com/wiki/P42_glider_shuttle}{pinball} & Dean Hickerson & 1989 \\
329 & $\lcm(7, 47)$ oscillator & Dean Hickerson, David Buckingham & $\leq$1992 \\
$856n$ & \href{https://conwaylife.com/wiki/Period-856_glider_gun}{p$856n$ glider loop}& David Buckingham & 1991
\end{longtable}
\end{center}
For the oscillators whose discovery year is marked with $\leq$, the exact discovery date was not recorded.

\newpage



\section*{\hfil Gallery \hfil}\label{sec:gallery}

On the right we provide the patterns in RLE format, a compact representation that can be pasted into almost any Life software (for example, \href{https://golly.sourceforge.io/}{Golly} or \href{https://lazyslug.com/lifeviewer/}{
LifeViewer}).

\patcols{(p1) Block}{\href{https://conwaylife.com/?rle=2o$2o!&name=(p1) Block}{\includegraphics[width=\textwidth]{p1}}}
{\detokenize{x = 2, y = 2, rule = B3/S23} \\
\detokenize{2o$2o!}}

\patcols{(p2) Blinker}{\href{https://conwaylife.com/?rle=3o!&name=(p2) Blinker}{\includegraphics[width=\textwidth]{p2}}}
{\detokenize{x = 3, y = 1, rule = B3/S23} \\
\detokenize{3o!}}

\patcols{(p3) Pulsar}{\href{https://conwaylife.com/?rle=2b3o3b3o2$o4bobo4bo$o4bobo4bo$o4bobo4bo$2b3o3b3o2$2b3o3b3o$o4bobo4bo$o4bobo4bo$o4bobo4bo2$2b3o3b3o!&name=(p3) Pulsar}{\includegraphics[width=\textwidth]{p3}}}
{\detokenize{x = 13, y = 13, rule = B3/S23
2b3o3b3o2$o4bobo4bo$o4bobo4bo$o4bobo4bo$2b3o3b3o2$2b3o3b3o$o4bobo4bo$o
4bobo4bo$o4bobo4bo2$2b3o3b3o!}}

\patcols{(p4) Pinwheel}{\href{https://conwaylife.com/?rle=6b2o$6b2o2$4b4o$2obo2bobo$2obobo2bo$3bo3b2ob2o$3bo4bob2o$4b4o2$4b2o$4b2o!&name=(p4) Pinwheel}{\includegraphics[width=\textwidth]{p4}}}
{\detokenize{x = 12, y = 12, rule = B3/S23
6b2o$6b2o2$4b4o$2obo2bobo$2obobo2bo$3bo3b2ob2o$3bo4bob2o$4b4o2$4b2o$4b
2o!}}

\patcols{(p5) Octagon 2}{\href{https://conwaylife.com/?rle=3b2o$2bo2bo$bo4bo$o6bo$o6bo$bo4bo$2bo2bo$3b2o!&name=(p5) Octagon 2}{\includegraphics[width=\textwidth]{p5}}}
{\detokenize{x = 8, y = 8, rule = B3/S23
3b2o$2bo2bo$bo4bo$o6bo$o6bo$bo4bo$2bo2bo$3b2o!}}

\patcols{(p6) \$rats}{\href{https://conwaylife.com/?rle=5b2o$6bo$4bo$2obob4o$2obo5bobo$3bo2b3ob2o$3bo4bo$4b3obo$7bo$6bo$6b2o!&name=(p6) $rats}{\includegraphics[width=\textwidth]{p6}}}
{\detokenize{x = 12, y = 11, rule = B3/S23
5b2o$6bo$4bo$2obob4o$2obo5bobo$3bo2b3ob2o$3bo4bo$4b3obo$7bo$6bo$6b2o!}}

\patcols{(p7) Burloaferimeter}{\href{https://conwaylife.com/?rle=4b2o$5bo$4bo$3bob3o$3bobo2bo$2obo3bobo$2obobo2bo$4b4o2$4b2o$4b2o!&name=(p7) Burloaferimeter}{\includegraphics[width=\textwidth]{p7}}}
{\detokenize{x = 10, y = 11, rule = B3/S23
4b2o$5bo$4bo$3bob3o$3bobo2bo$2obo3bobo$2obobo2bo$4b4o2$4b2o$4b2o!}}

\patcols{(p8) Figure eight}{\href{https://conwaylife.com/?rle=3o$3o$3o$3b3o$3b3o$3b3o!&name=(p8) Figure eight}{\includegraphics[width=\textwidth]{p8}}}
{\detokenize{x = 6, y = 6, rule = B3/S23} \\
\detokenize{3o$3o$3o$3b3o$3b3o$3b3o!}}

\patcols{(p9) Worker bee}{\href{https://conwaylife.com/?rle=2o12b2o$bo12bo$bobo8bobo$2b2o8b2o2$5b6o2$2b2o8b2o$bobo8bobo$bo12bo$2o12b2o!&name=(p9) Worker bee}{\includegraphics[width=\textwidth]{p9}}}
{\detokenize{x = 16, y = 11, rule = B3/S23
2o12b2o$bo12bo$bobo8bobo$2b2o8b2o2$5b6o2$2b2o8b2o$bobo8bobo$bo12bo$2o
12b2o!}}

\patcols{(p10) 42P10.3}{\href{https://conwaylife.com/?rle=6b2o$6b2o2$4b4o$3bo4bo$o2b5obo$3o6bo$3bo3bobob2o$2bo2bo4bobo$2bob2ob3o$3bo2bo$4bo2bo$5b2o!&name=(p10) 42P10.3}{\includegraphics[width=\textwidth]{p10}}}
{\detokenize{x = 13, y = 13, rule = B3/S23
6b2o$6b2o2$4b4o$3bo4bo$o2b5obo$3o6bo$3bo3bobob2o$2bo2bo4bobo$2bob2ob3o
$3bo2bo$4bo2bo$5b2o!}}

\patcols{(p11) 38P11.1}{\href{https://conwaylife.com/?rle=2b2ob2o$3bobobo$3bo4bo$2obo5bo$2obo6bo$3bobo5bo$3bob2o3b2o$4bo$5b7o$11bo$7b2o$7b2o!&name=(p11) 38P11.1}{\includegraphics[width=\textwidth]{p11}}}
{\detokenize{x = 12, y = 12, rule = B3/S23
2b2ob2o$3bobobo$3bo4bo$2obo5bo$2obo6bo$3bobo5bo$3bob2o3b2o$4bo$5b7o$11b
o$7b2o$7b2o!}}

\patcols{(p12) Dinner table}{\href{https://conwaylife.com/?rle=bo$b3o7b2o$4bo6bo$3b2o4bobo$9b2o2$5b3o$5b3o$2b2o$bobo4b2o$bo6bo$2o7b3o$11bo!&name=(p12) Dinner table}{\includegraphics[width=\textwidth]{p12}}}
{\detokenize{x = 13, y = 13, rule = B3/S23
bo$b3o7b2o$4bo6bo$3b2o4bobo$9b2o2$5b3o$5b3o$2b2o$bobo4b2o$bo6bo$2o7b3o$11bo!}}

\patcols{(p13) Buckingham's p13}{\href{https://conwaylife.com/?rle=4bo15bo$3bobo13bobo$3bobo13bobo$b3ob2o11b2ob3o$o23bo$b3ob2o11b2ob3o$3bob2o11b2obo$10bo3bo$11bobo$10b2ob2o$8bo2bobo2bo$7bobobobobobo$7b2o2bobo2b2o$11bobo$12bo!&name=Buckingham's p13}{\includegraphics[width=\textwidth]{p13}}}
{\detokenize{x = 25, y = 15, rule = B3/S23} \\
\detokenize{4bo15bo$3bobo13bobo$3bobo13bobo$b3ob2o11b2ob3o$o23bo$b3ob2o11b2ob3o$3b
ob2o11b2obo$10bo3bo$11bobo$10b2ob2o$8bo2bobo2bo$7bobobobobobo$7b2o2bo
bo2b2o$11bobo$12bo!}}

\patcols{(p14) Tumbler}{\href{https://conwaylife.com/?rle=2o3b2o$obobobo$obobobo$2bobo$b2ob2o$b2ob2o!&name=(p14) Tumbler}{\includegraphics[width=\textwidth]{p14}}}
{\detokenize{x = 7, y = 6, rule = B3/S23} \\
\detokenize{2o3b2o$obobobo$obobobo$2bobo$b2ob2o$b2ob2o!}}

\patcols{(p15) Pentadecathlon}{\href{https://conwaylife.com/?rle=2bo4bo$2ob4ob2o$2bo4bo!&name=(p15) Pentadecathlon}{\includegraphics[width=\textwidth]{p15}}}
{\detokenize{x = 10, y = 3, rule = B3/S23} \\
\detokenize{2bo4bo$2ob4ob2o$2bo4bo!}}

\patcols{(p16) Two pre-L hassler}{\href{https://conwaylife.com/?rle=4b2o12b2o$4o4b2o4b2o4b4o$3ob2o2b2o4b2o2b2ob3o$5bo12bo2$13b2o$13b2o$13b3o$8b3o$9b2o$9b2o2$5bo12bo$3ob2o2b2o4b2o2b2ob3o$4o4b2o4b2o4b4o$4b2o12b2o!&name=(p16) Two pre-L hassler}{\includegraphics[width=\textwidth]{p16}}}
{\detokenize{x = 24, y = 16, rule = B3/S23} \\
\detokenize{4b2o12b2o$4o4b2o4b2o4b4o$3ob2o2b2o4b2o2b2ob3o$5bo12bo2$13b2o$13b2o$13b
3o$8b3o$9b2o$9b2o2$5bo12bo$3ob2o2b2o4b2o2b2ob3o$4o4b2o4b2o4b4o$4b2o12b
2o!}}

\patcols{(p17) 54P17.1}{\href{https://conwaylife.com/?rle=5bo$4bobo$4bobo3b2o$b2obob2o3bo$2bobo6bob2o$o2bob2ob2obo2bo$2obobobo2bobo$3bob2ob2o2b2o$3bobo3bobo$4b2obobobo$6bobobo$6bobo$7b2o!&name=(p17) 54P17.1}{\includegraphics[width=\textwidth]{p17}}}
{\detokenize{x = 15, y = 13, rule = B3/S23} \\
\detokenize{5bo$4bobo$4bobo3b2o$b2obob2o3bo$2bobo6bob2o$o2bob2ob2obo2bo$2obobobo2b
obo$3bob2ob2o2b2o$3bobo3bobo$4b2obobobo$6bobobo$6bobo$7b2o!}}

\patcols{(p18) 117P18}{\href{https://conwaylife.com/?rle=37b2o$8b2o27bo$9bo25bobo$9bobo23b2o$10b2o14bo4bo$15bo4bo3b2ob4ob2o$13b2ob4ob2o3bo4bo$15bo4bo14b2o$10b2o23bobo$9bobo25bo$9bo27b2o$8b2o$23b2o$18b3o2b2o$18bo$18b2o$14b2o$14b2o$29b2o$2o27bo$bo25bobo$bobo23b2o$2b2o14bo4bo$7bo4bo3b2ob4ob2o$5b2ob4ob2o3bo4bo$7bo4bo14b2o$2b2o23bobo$bobo25bo$bo27b2o$2o!&name=(p18) 117P18}{\includegraphics[width=\textwidth]{p18}}}
{\detokenize{x = 39, y = 30, rule = B3/S23} \\
\detokenize{37b2o$8b2o27bo$9bo25bobo$9bobo23b2o$10b2o14bo4bo$15bo4bo3b2ob4ob2o$13b
2ob4ob2o3bo4bo$15bo4bo14b2o$10b2o23bobo$9bobo25bo$9bo27b2o$8b2o$23b2o
$18b3o2b2o$18bo$18b2o$14b2o$14b2o$29b2o$2o27bo$bo25bobo$bobo23b2o$2b2o
14bo4bo$7bo4bo3b2ob4ob2o$5b2ob4ob2o3bo4bo$7bo4bo14b2o$2b2o23bobo$bobo
25bo$bo27b2o$2o!}}

\patcols{(p19) Cribbage}{\href{https://conwaylife.com/?rle=4b2o$4bo$b2obo10bo$bo2b2o9b3o$3bo2bo11bo$bob4o10b2o$obo$o2b4o12b3o$b2o3bo11bo3bo6b2o$3b2o5b3o4bo5bo4bobo$3bo5bo3bo4bo3bo5bo$bobo4bo5bo4b3o5b2o$b2o6bo3bo11bo3b2o$10b3o12b4o2bo$29bobo$13b2o10b4obo$13bo11bo2bo$14b3o9b2o2bo$16bo10bob2o$27bo$26b2o!&name=(p19) Cribbage}{\includegraphics[width=\textwidth]{p19}}}
{\detokenize{x = 32, y = 21, rule = B3/S23} \\
\detokenize{4b2o$4bo$b2obo10bo$bo2b2o9b3o$3bo2bo11bo$bob4o10b2o$obo$o2b4o12b3o$b2o
3bo11bo3bo6b2o$3b2o5b3o4bo5bo4bobo$3bo5bo3bo4bo3bo5bo$bobo4bo5bo4b3o5b
2o$b2o6bo3bo11bo3b2o$10b3o12b4o2bo$29bobo$13b2o10b4obo$13bo11bo2bo$14b
3o9b2o2bo$16bo10bob2o$27bo$26b2o!}}

\patcols{(p20) octagon on HW emulator}{\href{https://conwaylife.com/?rle=3bo2bo$3bo2bo$b2ob2ob2o$3bo2bo$3bo2bo$b2ob2ob2o$3bo2bo$3bo2bo2$7b2o$2b2obo4bob2o$2bo10bo$3b2o6b2o$3o2b6o2b3o$o2bo8bo2bo$b2o10b2o!&name=(p20) octagon on HW emulator}{\includegraphics[width=\textwidth]{p20}}}
{\detokenize{x = 16, y = 16, rule = B3/S23} \\
\detokenize{3bo2bo$3bo2bo$b2ob2ob2o$3bo2bo$3bo2bo$b2ob2ob2o$3bo2bo$3bo2bo2$7b2o$2b
2obo4bob2o$2bo10bo$3b2o6b2o$3o2b6o2b3o$o2bo8bo2bo$b2o10b2o!}}

\patcols{(p21) jam on 44P7.2}{\href{https://conwaylife.com/?rle=10b2o$4bo4bo2bo$3bobo3bobo$3bobo4bo3b2o$2obob2o4b2obo$ob2o4bo3bo$4b3obo3bo$4bo$5b3o2b2o$8bo$5b2obobobo$5b2obobob3o$9b2o4bo$11b4o$11bo$12bo$11b2o!&name=(p21) jam on 44P7.2}{\includegraphics[width=\textwidth]{p21}}}
{\detokenize{x = 16, y = 17, rule = B3/S23} \\
\detokenize{10b2o$4bo4bo2bo$3bobo3bobo$3bobo4bo3b2o$2obob2o4b2obo$ob2o4bo3bo$4b3ob
o3bo$4bo$5b3o2b2o$8bo$5b2obobobo$5b2obobob3o$9b2o4bo$11b4o$11bo$12bo$
11b2o!}}

\patcols{(p22) 168P22.1}{\href{https://conwaylife.com/?rle=12b2o5b2o$7b2o3bo7bo3b2o$6bobo4bo5bo4bobo$8bob3obo3bob3obo$5b2obobo3bo3bo3bobob2o$8bo5bobobo5bo$7bobo2b2obobob2o2bobo$8bo3bo2bobo2bo3bo$4b2o8b2ob2o8b2o$2o2bo23bo2b2o$obobo5b3o7b3o5bobobo$2bob2o4bobo7bobo4b2obo$obobo5b3o7b3o5bobobo$2o2bo23bo2b2o$4b2o8b2ob2o8b2o$8bo3bo2bobo2bo3bo$7bobo2b2obobob2o2bobo$8bo5bobobo5bo$5b2obobo3bo3bo3bobob2o$8bob3obo3bob3obo$6bobo4bo5bo4bobo$7b2o3bo7bo3b2o$12b2o5b2o!&name=(p22) 168P22.1}{\includegraphics[width=\textwidth]{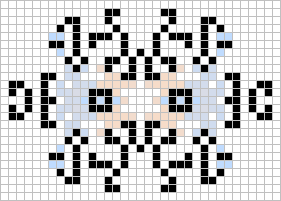}}}
{\detokenize{x = 33, y = 23, rule = B3/S23} \\
\detokenize{12b2o5b2o$7b2o3bo7bo3b2o$6bobo4bo5bo4bobo$8bob3obo3bob3obo$5b2obobo3b
o3bo3bobob2o$8bo5bobobo5bo$7bobo2b2obobob2o2bobo$8bo3bo2bobo2bo3bo$4b
2o8b2ob2o8b2o$2o2bo23bo2b2o$obobo5b3o7b3o5bobobo$2bob2o4bobo7bobo4b2o
bo$obobo5b3o7b3o5bobobo$2o2bo23bo2b2o$4b2o8b2ob2o8b2o$8bo3bo2bobo2bo3b
o$7bobo2b2obobob2o2bobo$8bo5bobobo5bo$5b2obobo3bo3bo3bobob2o$8bob3obo
3bob3obo$6bobo4bo5bo4bobo$7b2o3bo7bo3b2o$12b2o5b2o!}}

\patcols{(p23) David Hilbert}{\href{https://conwaylife.com/?rle=7b2o15b2o$8bo15bo$6bo19bo$6b5o11b5o$10bo11bo$4b4o16b5o$4bo2bo17bo2bo$7bo2bo14b2o$7bo3bo12b3o$7bo3bo2b2ob2o5b2o$8bo2b2ob2ob2o5b2o$10bo12bo$8b2o$7b3o13bo$3b2o2b3o6b2o4bobo3b2o$3bo11bo2bo3bobo4bo$2obo12b2o4b3o4bob2o$ob2ob2o19b2ob2obo$5bo21bo$5bobo17bobo$6b2o17b2o$10bo11bo$6b5o11b5o$6bo19bo$8bo15bo$7b2o15b2o!&name=(p23) David Hilbert}{\includegraphics[width=\textwidth]{p23}}}
{\detokenize{x = 33, y = 26, rule = B3/S23} \\
\detokenize{7b2o15b2o$8bo15bo$6bo19bo$6b5o11b5o$10bo11bo$4b4o16b5o$4bo2bo17bo2bo$
7bo2bo14b2o$7bo3bo12b3o$7bo3bo2b2ob2o5b2o$8bo2b2ob2ob2o5b2o$10bo12bo$
8b2o$7b3o13bo$3b2o2b3o6b2o4bobo3b2o$3bo11bo2bo3bobo4bo$2obo12b2o4b3o4b
ob2o$ob2ob2o19b2ob2obo$5bo21bo$5bobo17bobo$6b2o17b2o$10bo11bo$6b5o11b
5o$6bo19bo$8bo15bo$7b2o15b2o!}}

\patcols{(p24) pulsar on figure eight}{\href{https://conwaylife.com/?rle=4bo5bo$4bo5bo7b2o$4b2o3b2o$16bo3bo$3o2b2ob2o2b3obo4bo$2bobobobobobo5bobobo$4b2o3b2o8bobobo$20bo4bo$4b2o3b2o10bo3bo$2bobobobobobo$3o2b2ob2o2b3o7b2o2$4b2o3b2o$4bo5bo$4bo5bo!&name=(p24) pulsar on figure eight}{\includegraphics[width=\textwidth]{p24}}}
{\detokenize{x = 26, y = 15, rule = B3/S23} \\
\detokenize{4bo5bo$4bo5bo7b2o$4b2o3b2o$16bo3bo$3o2b2ob2o2b3obo4bo$2bobobobobobo5bo
bobo$4b2o3b2o8bobobo$20bo4bo$4b2o3b2o10bo3bo$2bobobobobobo$3o2b2ob2o2b
3o7b2o2$4b2o3b2o$4bo5bo$4bo5bo!}}

\patcols{(p25) 134P25}{\href{https://conwaylife.com/?rle=11b2o$7b2o2bo$6bo2bobo$2b2obob3ob2o$2bob2obobobo$13bo$3b4obob3obo$2bo3b2ob2o3bo$3b3ob3ob3o$5bo5bo7bo$17b3o$6bo3bo5bo$8bo7b2o$2o$bo$bobo9b3o$2b2o9bobo$7b3o3b3o$7bobo9b2o$7b3o9bobo$21bo$21b2o$5b2o7bo$6bo5bo3bo$3b3o$3bo7bo5bo$9b3ob3ob3o$8bo3b2ob2o3bo$8bob3obob4o$9bo$11bobobob2obo$10b2ob3obob2o$11bobo2bo$11bo2b2o$10b2o!&name=(p25) 134P25}{\includegraphics[width=\textwidth]{p25}}}
{\detokenize{x = 23, y = 35, rule = B3/S23} \\
\detokenize{11b2o$7b2o2bo$6bo2bobo$2b2obob3ob2o$2bob2obobobo$13bo$3b4obob3obo$2bo
3b2ob2o3bo$3b3ob3ob3o$5bo5bo7bo$17b3o$6bo3bo5bo$8bo7b2o$2o$bo$bobo9b3o
$2b2o9bobo$7b3o3b3o$7bobo9b2o$7b3o9bobo$21bo$21b2o$5b2o7bo$6bo5bo3bo$
3b3o$3bo7bo5bo$9b3ob3ob3o$8bo3b2ob2o3bo$8bob3obob4o$9bo$11bobobob2obo$
10b2ob3obob2o$11bobo2bo$11bo2b2o$10b2o!}}

\patcols{(p26) p26 pre-pulsar shuttle}{\href{https://conwaylife.com/?rle=16b2o3b2o$15bo2bobo2bo$11b2o3b2o3b2o3b2o$11bo15bo$8b2obo15bob2o$7bobob2o13b2obobo$7bobo5b3o3b3o5bobo$5b2o2bo5bobo3bobo5bo2b2o$4bo4b2o4b3o3b3o4b2o4bo$4b5o21b5o$8bo21bo$2b4o27b4o$2bo2bo27bo2bo2$16b2o3b2o$bo4b3o8bo3bo8b3o4bo$obo3bobo5bo9bo5bobo3bobo$obo3b3o5b2o7b2o5b3o3bobo$bo35bo2$bo35bo$obo3b3o5b2o7b2o5b3o3bobo$obo3bobo5bo9bo5bobo3bobo$bo4b3o8bo3bo8b3o4bo$16b2o3b2o2$2bo2bo27bo2bo$2b4o27b4o$8bo21bo$4b5o21b5o$4bo4b2o4b3o3b3o4b2o4bo$5b2o2bo5bobo3bobo5bo2b2o$7bobo5b3o3b3o5bobo$7bobob2o13b2obobo$8b2obo15bob2o$11bo15bo$11b2o3b2o3b2o3b2o$15bo2bobo2bo$16b2o3b2o!&name=p26 pre-pulsar shuttle}{\includegraphics[width=\textwidth]{p26}}}
{\detokenize{x = 39, y = 39, rule = B3/S23} \\
\detokenize{16b2o3b2o$15bo2bobo2bo$11b2o3b2o3b2o3b2o$11bo15bo$8b2obo15bob2o$7bobo
b2o13b2obobo$7bobo5b3o3b3o5bobo$5b2o2bo5bobo3bobo5bo2b2o$4bo4b2o4b3o3b
3o4b2o4bo$4b5o21b5o$8bo21bo$2b4o27b4o$2bo2bo27bo2bo2$16b2o3b2o$bo4b3o
8bo3bo8b3o4bo$obo3bobo5bo9bo5bobo3bobo$obo3b3o5b2o7b2o5b3o3bobo$bo35b
o2$bo35bo$obo3b3o5b2o7b2o5b3o3bobo$obo3bobo5bo9bo5bobo3bobo$bo4b3o8bo
3bo8b3o4bo$16b2o3b2o2$2bo2bo27bo2bo$2b4o27b4o$8bo21bo$4b5o21b5o$4bo4b
2o4b3o3b3o4b2o4bo$5b2o2bo5bobo3bobo5bo2b2o$7bobo5b3o3b3o5bobo$7bobob2o
13b2obobo$8b2obo15bob2o$11bo15bo$11b2o3b2o3b2o3b2o$15bo2bobo2bo$16b2o
3b2o!}}

\patcols{(p27) 123P27.1}{\href{https://conwaylife.com/?rle=b2ob2o$2bobobob2o$bo4bobo$ob4o4bo$o4bo3b4obo$b4o3bo4b2o$3bo4b4o3b2o$5bobo4b2o3bo$4b2obo2b2o2b2obo$3bo3bobo2bobobo$4b3o2b2obo6bo$7b2o3bo2b5o$6bo2bob2obobo$5bob2o3bo2bo2b2o$6bo2b2o3bobo2bo$7b2o3bobob2o$9b4obobo$9bo2bobobo$10b5ob2o$15bo2bo$9bob4o2b2o$9b2obo!&name=(p27) 123P27.1}{\includegraphics[width=\textwidth]{p27}}}
{\detokenize{x = 20, y = 22, rule = B3/S23} \\
\detokenize{b2ob2o$2bobobob2o$bo4bobo$ob4o4bo$o4bo3b4obo$b4o3bo4b2o$3bo4b4o3b2o$5b
obo4b2o3bo$4b2obo2b2o2b2obo$3bo3bobo2bobobo$4b3o2b2obo6bo$7b2o3bo2b5o$
6bo2bob2obobo$5bob2o3bo2bo2b2o$6bo2b2o3bobo2bo$7b2o3bobob2o$9b4obobo$
9bo2bobobo$10b5ob2o$15bo2bo$9bob4o2b2o$9b2obo!}}

\patcols{(p28) newshuttle}{\href{https://conwaylife.com/?rle=26b2o$20bo3bo2bo2bo$18b3o3b3o3b3o$8b2o7bo15bo7b2o$9bo7b2o5b3o5b2o7bo$9bobo11bo3bo11bobo$10b2o11b2ob2o11b2o2$3bo43bo$3b3o39b3o$6bo7b3o17b3o7bo$5b2o7bobo17bobo7b2o$14b3o3b3o5b3o3b3o$20bobo5bobo$10b3o7b3o5b3o7b3o$10bobo25bobo$10b3o25b3o$3b2o18b2ob2o18b2o$2bobo18bo3bo18bobo$2bo21b3o21bo$b2o9b3o11bobo7b3o9b2o$12bobo12b2o7bobo$12b3o5b2o14b3o$2o3b2o10b2o2bo10b2o10b2o$obobobo10bob2o10bobo10bobob2o$2bobo14bo11bo14bobo$b2obobo10bobo10b2obo10bobobobo$5b2o10b2o10bo2b2o10b2o3b2o$12b3o14b2o5b3o$12bobo7b2o12bobo$b2o9b3o7bobo11b3o9b2o$2bo21b3o21bo$2bobo18bo3bo18bobo$3b2o18b2ob2o18b2o$10b3o25b3o$10bobo25bobo$10b3o7b3o5b3o7b3o$20bobo5bobo$14b3o3b3o5b3o3b3o$5b2o7bobo17bobo7b2o$6bo7b3o17b3o7bo$3b3o39b3o$3bo43bo2$10b2o11b2ob2o11b2o$9bobo11bo3bo11bobo$9bo7b2o5b3o5b2o7bo$8b2o7bo15bo7b2o$18b3o3b3o3b3o$20bo2bo2bo3bo$23b2o!&name=(p28) newshuttle}{\includegraphics[width=\textwidth]{p28}}}
{\detokenize{x = 51, y = 51, rule = B3/S23} \\
\detokenize{26b2o$20bo3bo2bo2bo$18b3o3b3o3b3o$8b2o7bo15bo7b2o$9bo7b2o5b3o5b2o7bo$
9bobo11bo3bo11bobo$10b2o11b2ob2o11b2o2$3bo43bo$3b3o39b3o$6bo7b3o17b3o
7bo$5b2o7bobo17bobo7b2o$14b3o3b3o5b3o3b3o$20bobo5bobo$10b3o7b3o5b3o7b
3o$10bobo25bobo$10b3o25b3o$3b2o18b2ob2o18b2o$2bobo18bo3bo18bobo$2bo21b
3o21bo$b2o9b3o11bobo7b3o9b2o$12bobo12b2o7bobo$12b3o5b2o14b3o$2o3b2o10b
2o2bo10b2o10b2o$obobobo10bob2o10bobo10bobob2o$2bobo14bo11bo14bobo$b2ob
obo10bobo10b2obo10bobobobo$5b2o10b2o10bo2b2o10b2o3b2o$12b3o14b2o5b3o$
12bobo7b2o12bobo$b2o9b3o7bobo11b3o9b2o$2bo21b3o21bo$2bobo18bo3bo18bobo
$3b2o18b2ob2o18b2o$10b3o25b3o$10bobo25bobo$10b3o7b3o5b3o7b3o$20bobo5bo
bo$14b3o3b3o5b3o3b3o$5b2o7bobo17bobo7b2o$6bo7b3o17b3o7bo$3b3o39b3o$3bo
43bo2$10b2o11b2ob2o11b2o$9bobo11bo3bo11bobo$9bo7b2o5b3o5b2o7bo$8b2o7bo
15bo7b2o$18b3o3b3o3b3o$20bo2bo2bo3bo$23b2o!}}

\patcols{(p29) p29 pre-pulsar shuttle}{\href{https://conwaylife.com/?rle=15bo$13b3o$12bo$12b2o2$bo$obo6b3o$bo7bobo3b2o$9b3o3b2o3$19b2o$9b3o7b2o5b2o$bo7bobo14bo$obo6b3o12bobo$bo22b2o$13b3o3b3o$13bobo3bobo$13b3o3b3o7$13bo7bo$12bobo5bobo$13bo7bo!&name=p29 pre-pulsar shuttle}{\includegraphics[width=\textwidth]{p29}}}
{\detokenize{x = 28, y = 28, rule = B3/S23} \\
\detokenize{15bo$13b3o$12bo$12b2o2$bo$obo6b3o$bo7bobo3b2o$9b3o3b2o3$19b2o$9b3o7b2o
5b2o$bo7bobo14bo$obo6b3o12bobo$bo22b2o$13b3o3b3o$13bobo3bobo$13b3o3b3o
7$13bo7bo$12bobo5bobo$13bo7bo!}}

\patcols{(p30) queen bee shuttle}{\href{https://conwaylife.com/?rle=9b2o$9bobo$4b2o6bo7b2o$2obo2bo2bo2bo7b2o$2o2b2o6bo$9bobo$9b2o!&name=(p30) queen bee shuttle}{\includegraphics[width=\textwidth]{p30}}}
{\detokenize{x = 7, y = 22, rule = B3/S23} \\
\detokenize{9b2o$9bobo$4b2o6bo7b2o$2obo2bo2bo2bo7b2o$2o2b2o6bo$9bobo$9b2o!}}

\patcols{(p31) Merzenich's p31}{\href{https://conwaylife.com/?rle=7b2obo2bob2o$2o4bo2bo4bo2bo4b2o$2o5bobo4bobo5b2o$8bo6bo6$8bo6bo$2o5bobo4bobo5b2o$2o4bo2bo4bo2bo4b2o$7b2obo2bob2o!&name=Merzenich's p31}{\includegraphics[width=\textwidth]{p31}}}
{\detokenize{x = 24, y = 13, rule = B3/S23} \\
\detokenize{7b2obo2bob2o$2o4bo2bo4bo2bo4b2o$2o5bobo4bobo5b2o$8bo6bo6$8bo6bo$2o5bob
o4bobo5b2o$2o4bo2bo4bo2bo4b2o$7b2obo2bob2o!}}

\patcols{(p32) gourmet}{\href{https://conwaylife.com/?rle=10b2o$10bo$4b2ob2obo4b2o$2bo2bobobo5bo$2b2o4bo8bo$16b2o2$16b2o$o9b3o2bobo$3o7bobo3bo$3bo6bobo4b3o$2bobo14bo$2b2o2$2b2o$2bo8bo4b2o$4bo5bobobo2bo$3b2o4bob2ob2o$9bo$8b2o!&name=(p32) gourmet}{\includegraphics[width=\textwidth]{p32}}}
{\detokenize{x = 20, y = 20, rule = B3/S23} \\
\detokenize{10b2o$10bo$4b2ob2obo4b2o$2bo2bobobo5bo$2b2o4bo8bo$16b2o2$16b2o$o9b3o2b
obo$3o7bobo3bo$3bo6bobo4b3o$2bobo14bo$2b2o2$2b2o$2bo8bo4b2o$4bo5bobob
o2bo$3b2o4bob2ob2o$9bo$8b2o!}}

\patcols{(p33) 258P3 on Achim's p11}{\href{https://conwaylife.com/?rle=19b2o$19b2o$14bo10bo$13bobo8bobo$12bo2bo3b2o3bo2bo$11bo6b4o6bo$12b2o4bo2bo4b2o3$14b2o8b2o$9b2o2b2o10b2o2b2o$9b2o2b2o10b2o2b2o$14b2o8b2o3$12b2o4bo2bo4b2o$11bo6b4o6bo$12bo2bo3b2o3bo2bo$13bobo8bobo$14bo10bo$19b2o$14bo4b2o3$3bo3bobob6obobo3bo$2bob2obo2bobo2bobo2bob2obo$2bo7b2o4b2o7bo$b2o2bo3bo2b4o2bo3bo2b2o$o2b6ob2o4b2ob6o2bo$2o6b2ob6ob2o6b2o$2b2obo16bob2o$2o2bo8b2o8bo2b2o$bobo2b3o3bo2bo3b3o2bobo$o2b3o5b2o2b2o5b3o2bo$b2o3bob2obo4bob2obo3b2o$2bobo2bob3o4b3obo2bobo$2bob2obo3b6o3bob2obo$b2o2bobobobo4bobobobo2b2o$o2bo7b6o7bo2bo$b2o3bo14bo3b2o$5bo3b4o2b4o3bo$4bob2obo2bo2bo2bob2obo$4bo2b3o8b3o2bo$b2obo2b3o2bo2bo2b3o2bob2o$b2obob2o4b4o4b2obob2o$5bo3bo8bo3bo$6bobo10bobo$4bobobobo6bobobobo$4b2o3b2o6b2o3b2o!&name=(p33) 258P3 on Achim's p11}{\includegraphics[width=\textwidth]{p33}}}
{\detokenize{x = 31, y = 49, rule = B3/S23} \\
\detokenize{19b2o$19b2o$14bo10bo$13bobo8bobo$12bo2bo3b2o3bo2bo$11bo6b4o6bo$12b2o4b
o2bo4b2o3$14b2o8b2o$9b2o2b2o10b2o2b2o$9b2o2b2o10b2o2b2o$14b2o8b2o3$12b
2o4bo2bo4b2o$11bo6b4o6bo$12bo2bo3b2o3bo2bo$13bobo8bobo$14bo10bo$19b2o$
14bo4b2o3$3bo3bobob6obobo3bo$2bob2obo2bobo2bobo2bob2obo$2bo7b2o4b2o7bo
$b2o2bo3bo2b4o2bo3bo2b2o$o2b6ob2o4b2ob6o2bo$2o6b2ob6ob2o6b2o$2b2obo16b
ob2o$2o2bo8b2o8bo2b2o$bobo2b3o3bo2bo3b3o2bobo$o2b3o5b2o2b2o5b3o2bo$b2o
3bob2obo4bob2obo3b2o$2bobo2bob3o4b3obo2bobo$2bob2obo3b6o3bob2obo$b2o2b
obobobo4bobobobo2b2o$o2bo7b6o7bo2bo$b2o3bo14bo3b2o$5bo3b4o2b4o3bo$4bob
2obo2bo2bo2bob2obo$4bo2b3o8b3o2bo$b2obo2b3o2bo2bo2b3o2bob2o$b2obob2o4b
4o4b2obob2o$5bo3bo8bo3bo$6bobo10bobo$4bobobobo6bobobobo$4b2o3b2o6b2o3b
2o!}}

\patcols{(p34) 74P34}{\href{https://conwaylife.com/?rle=15bo$13b3o$3b2o7bo$4bo7b2o$4bobo$5b2o2$2o$bo$bobo$2b2o11bo$15bo9b2o$9b3o4bo8bo$9bobo11bobo$9bobo11b2o2$2b2o11bobo$bobo11bobo$bo8bo4b3o$2o9bo$11bo11b2o$23bobo$25bo$25b2o2$20b2o$20bobo$13b2o7bo$14bo7b2o$11b3o$11bo!&name=(p34) 74P34}{\includegraphics[width=\textwidth]{p34}}}
{\detokenize{x = 27, y = 31, rule = B3/S23} \\
\detokenize{15bo$13b3o$3b2o7bo$4bo7b2o$4bobo$5b2o2$2o$bo$bobo$2b2o11bo$15bo9b2o$9b
3o4bo8bo$9bobo11bobo$9bobo11b2o2$2b2o11bobo$bobo11bobo$bo8bo4b3o$2o9bo
$11bo11b2o$23bobo$25bo$25b2o2$20b2o$20bobo$13b2o7bo$14bo7b2o$11b3o$11b
o!}}

\patcols{(p35) toaster on 44P7.2}{\href{https://conwaylife.com/?rle=16bo6b2o$15bobob2o2bo$4bo10bobobobobo$3bobo8b2obo3bob2o$3bobo6bo3b2obob2o3bo$2obob2o8bo7bo$ob2o4bo6bo7bo$4b3obo3bo3b2obob2o3bo$4bo9b2obo3bob2o$5b3o2b2o3bobobobobo$8bo6bobob2o2bo$5b2obobobo3bo6b2o$5b2obobob4o$9b2o$11b3o$11bo2bo$13b2o!&name=(p35) toaster on 44P7.2}{\includegraphics[width=\textwidth]{p35}}}
{\detokenize{x = 25, y = 17, rule = B3/S23} \\
\detokenize{16bo6b2o$15bobob2o2bo$4bo10bobobobobo$3bobo8b2obo3bob2o$3bobo6bo3b2obo
b2o3bo$2obob2o8bo7bo$ob2o4bo6bo7bo$4b3obo3bo3b2obob2o3bo$4bo9b2obo3bob
2o$5b3o2b2o3bobobobobo$8bo6bobob2o2bo$5b2obobobo3bo6b2o$5b2obobob4o$9b
2o$11b3o$11bo2bo$13b2o!}}

\patcols{(p36) p36 toad hassler}{\href{https://conwaylife.com/?rle=2o16b2o$bo7b2o7bo$bobo3bo4bo3bobo$2b2o2bo6bo2b2o$5bo8bo$5bo8bo$5bo8bo$2b2o2bo6bo2b2o$bobo3bo4bo3bobo$bo7b2o7bo$2o16b2o$11bo$9bo2bo$9bo2bo$10bo$2o16b2o$bo7b2o7bo$bobo3bo4bo3bobo$2b2o2bo6bo2b2o$5bo8bo$5bo8bo$5bo8bo$2b2o2bo6bo2b2o$bobo3bo4bo3bobo$bo7b2o7bo$2o16b2o!&name=p36 toad hassler}{\includegraphics[width=\textwidth]{p36}}}
{\detokenize{x = 20, y = 26, rule = B3/S23} \\
\detokenize{2o16b2o$bo7b2o7bo$bobo3bo4bo3bobo$2b2o2bo6bo2b2o$5bo8bo$5bo8bo$5bo8bo$
2b2o2bo6bo2b2o$bobo3bo4bo3bobo$bo7b2o7bo$2o16b2o$11bo$9bo2bo$9bo2bo$
10bo$2o16b2o$bo7b2o7bo$bobo3bo4bo3bobo$2b2o2bo6bo2b2o$5bo8bo$5bo8bo$5b
o8bo$2b2o2bo6bo2b2o$bobo3bo4bo3bobo$bo7b2o7bo$2o16b2o!}}

\patcols{(p37) Beluchenko's p37}{\href{https://conwaylife.com/?rle=11b2o11b2o$11b2o11b2o3$6bo23bo$5bobo6b3o3b3o6bobo$4bo2bo5bo9bo5bo2bo$5b2o6bo2bo3bo2bo6b2o$14b2o5b2o3$2o33b2o$2o33b2o$6b2o21b2o$5bo2bo19bo2bo$5bo2bo19bo2bo$5bobo21bobo4$5bobo21bobo$5bo2bo19bo2bo$5bo2bo19bo2bo$6b2o21b2o$2o33b2o$2o33b2o3$14b2o5b2o$5b2o6bo2bo3bo2bo6b2o$4bo2bo5bo9bo5bo2bo$5bobo6b3o3b3o6bobo$6bo23bo3$11b2o11b2o$11b2o11b2o!&name=Beluchenko's p37}{\includegraphics[width=\textwidth]{p37}}}
{\detokenize{x = 37, y = 37, rule = B3/S23} \\
\detokenize{11b2o11b2o$11b2o11b2o3$6bo23bo$5bobo6b3o3b3o6bobo$4bo2bo5bo9bo5bo2bo$
5b2o6bo2bo3bo2bo6b2o$14b2o5b2o3$2o33b2o$2o33b2o$6b2o21b2o$5bo2bo19bo2b
o$5bo2bo19bo2bo$5bobo21bobo4$5bobo21bobo$5bo2bo19bo2bo$5bo2bo19bo2bo$
6b2o21b2o$2o33b2o$2o33b2o3$14b2o5b2o$5b2o6bo2bo3bo2bo6b2o$4bo2bo5bo9bo
5bo2bo$5bobo6b3o3b3o6bobo$6bo23bo3$11b2o11b2o$11b2o11b2o!}}

\patcols{(p38) Raucci's p38}{\href{https://conwaylife.com/?rle=4bo32bo$4bo32bo$4bo32bo$6b2o3bo18bo3b2o$3o3b2o3b2o16b2o3b2o3b3o$13bo14bo$12b2o14b2o3$6b2o26b2o$6bo28bo$7b3o22b3o$9bo22bo!&name=(p38) Raucci's p38}{\includegraphics[width=\textwidth]{p38}}}
{\detokenize{x = 42, y = 13, rule = B3/S23} \\
\detokenize{4bo32bo$4bo32bo$4bo32bo$6b2o3bo18bo3b2o$3o3b2o3b2o16b2o3b2o3b3o$13bo
14bo$12b2o14b2o3$6b2o26b2o$6bo28bo$7b3o22b3o$9bo22bo!}}

\patcols{(p39) 134P39.1}{\href{https://conwaylife.com/?rle=7b2o3b2o$7bo3bobo$4b2obo2bo2bobo$4bo2b2o3bo2b3o$6bo3bo2bo4bo$7b3obob2ob3o$9bo5bo$10bo4bob2o$4b2o5b2ob2ob2o$2obobo3bo2bobo$2obob3ob2obobo10bo$3bo2b2o3bobo10bobo$3bobo2b2o2bo11bobo$4bobo3b2o11b2ob2o$6bobo$6bobob2o11b2ob2o$7b2ob2o11b2obo$15bo3bo8bo$16bobo8b2o$15b2ob2o$13bo2bobo2bo$12bobobobobobo$12b2o2bobo2b2o$16bobo$17bo!&name=(p39) 134P39.1}{\includegraphics[width=\textwidth]{p39}}}
{\detokenize{x = 29, y = 25, rule = B3/S23} \\
\detokenize{7b2o3b2o$7bo3bobo$4b2obo2bo2bobo$4bo2b2o3bo2b3o$6bo3bo2bo4bo$7b3obob2o
b3o$9bo5bo$10bo4bob2o$4b2o5b2ob2ob2o$2obobo3bo2bobo$2obob3ob2obobo10bo
$3bo2b2o3bobo10bobo$3bobo2b2o2bo11bobo$4bobo3b2o11b2ob2o$6bobo$6bobob
2o11b2ob2o$7b2ob2o11b2obo$15bo3bo8bo$16bobo8b2o$15b2ob2o$13bo2bobo2bo$
12bobobobobobo$12b2o2bobo2b2o$16bobo$17bo!}}

\patcols{(p40) p40 B-heptomino shuttle}{\href{https://conwaylife.com/?rle=b2o20b2o$b2o20bobo$24b3o$25b2o$3o19b2o$2o20b3o$3b2o$2b3o$bobo3b2o2b2o10b2o$b2o3bo8b2o6b2o$7b2o7b2o$15b2o$11bo3bo$11bo$7bo4bo5bo$6bobo8bobo$7bo10bo$11bo2bo$9bo2b2o2bo$9bo2b2o2bo$11b4o$9bobo2bobo$9b2o4b2o!&name=p40 B-heptomino shuttle}{\includegraphics[width=\textwidth]{p40}}}
{\detokenize{x = 27, y = 23, rule = B3/S23} \\
\detokenize{b2o20b2o$b2o20bobo$24b3o$25b2o$3o19b2o$2o20b3o$3b2o$2b3o$bobo3b2o2b2o
10b2o$b2o3bo8b2o6b2o$7b2o7b2o$15b2o$11bo3bo$11bo$7bo4bo5bo$6bobo8bobo$
7bo10bo$11bo2bo$9bo2b2o2bo$9bo2b2o2bo$11b4o$9bobo2bobo$9b2o4b2o!}}

\patcols{(p41) 204P41}{\href{https://conwaylife.com/?rle=34bo$32b3o$31bo$16b2o13b2o$17bo$17bobo$18b2o$29bo$28bobo$27bo3bo$27bo3bo$27bo3bo$28bobo$29bo2$5b2o24b2o$6bo24bo8bobo$4bo14b2o11b3o6b2o$2b4o10b2o2bo13bo6bo$bo12bob2o3b4o$o2b3o8b2o3bo3b2obo$b2o2bo8bob2o3bo3b2o$3b2o5b2o4b4o3b2obo$3bo5b2o9bo2b2o$4bo6bo8b2o$b3o$bo5b2o15bo27bo$8bo14bobo24bobo$5b3o15b2o26b2o$5bo2$67bo$20b2o26b2o15b3o$20bobo24bobo14bo$20bo27bo15b2o5bo$69b3o$51b2o8bo6bo$48b2o2bo9b2o5bo$46bob2o3b4o4b2o5b2o$46b2o3bo3b2obo8bo2b2o$46bob2o3bo3b2o8b3o2bo$48b4o3b2obo12bo$31bo6bo13bo2b2o10b4o$30b2o6b3o11b2o14bo$30bobo8bo24bo$40b2o24b2o2$43bo$42bobo$41bo3bo$41bo3bo$41bo3bo$42bobo$43bo$53b2o$53bobo$55bo$40b2o13b2o$41bo$38b3o$38bo!&name=(p41) 204P41}{\includegraphics[width=\textwidth]{p41}}}
{\detokenize{x = 73, y = 61, rule = B3/S23} \\
\detokenize{34bo$32b3o$31bo$16b2o13b2o$17bo$17bobo$18b2o$29bo$28bobo$27bo3bo$27bo
3bo$27bo3bo$28bobo$29bo2$5b2o24b2o$6bo24bo8bobo$4bo14b2o11b3o6b2o$2b4o
10b2o2bo13bo6bo$bo12bob2o3b4o$o2b3o8b2o3bo3b2obo$b2o2bo8bob2o3bo3b2o$
3b2o5b2o4b4o3b2obo$3bo5b2o9bo2b2o$4bo6bo8b2o$b3o$bo5b2o15bo27bo$8bo14b
obo24bobo$5b3o15b2o26b2o$5bo2$67bo$20b2o26b2o15b3o$20bobo24bobo14bo$
20bo27bo15b2o5bo$69b3o$51b2o8bo6bo$48b2o2bo9b2o5bo$46bob2o3b4o4b2o5b2o
$46b2o3bo3b2obo8bo2b2o$46bob2o3bo3b2o8b3o2bo$48b4o3b2obo12bo$31bo6bo
13bo2b2o10b4o$30b2o6b3o11b2o14bo$30bobo8bo24bo$40b2o24b2o2$43bo$42bobo
$41bo3bo$41bo3bo$41bo3bo$42bobo$43bo$53b2o$53bobo$55bo$40b2o13b2o$41bo
$38b3o$38bo!}}

\patcols{(p42) unix on 44P7.2}{\href{https://conwaylife.com/?rle=16b2o$16bo2bo2$4bo$3bobo10bob2o$3bobo5b2o2bobo$2obob2o4bo4bo$ob2o4bo6bo$4b3obo3bo2bo$4bo$5b3o2b2o$8bo$5b2obobobo$5b2obobob3o$9b2o4bo$11b4o$11bo$12bo$11b2o!&name=(p42) unix on 44P7.2}{\includegraphics[width=\textwidth]{p42}}}
{\detokenize{x = 20, y = 19, rule = B3/S23} \\
\detokenize{16b2o$16bo2bo2$4bo$3bobo10bob2o$3bobo5b2o2bobo$2obob2o4bo4bo$ob2o4bo6b
o$4b3obo3bo2bo$4bo$5b3o2b2o$8bo$5b2obobobo$5b2obobob3o$9b2o4bo$11b4o$
11bo$12bo$11b2o!}}

For periods 20, 21, and 24, the precise form of the first LCM oscillators is unknown: the patterns provided are the earliest that could have been constructed.

\printbibliography

@periodical{lifeline:one,
  title = {{LIFELINE}: A Quarterly Newsletter for Enthusiasts of John Conway's Game of Life},
  editor = {Wainwright, Robert T.},
  url = {https://conwaylife.com/wiki/Lifeline_Volume_1},
  date = {1971-03},
  number = {Volume 1},
}

@periodical{lifeline:three,
  title = {{LIFELINE}: A Quarterly Newsletter for Enthusiasts of John Conway's Game of Life},
  editor = {Wainwright, Robert T.},
  url = {https://conwaylife.com/wiki/Lifeline_Volume_3},
  date = {1971-09},
  number = {Volume 3},
}

@periodical{lifeline:five,
  title = {{LIFELINE}: A Quarterly Newsletter for Enthusiasts of John Conway's Game of Life},
  editor = {Wainwright, Robert T.},
  url = {https://conwaylife.com/wiki/Lifeline_Volume_5},
  date = {1972-09},
  number = {Volume 5},
}

@periodical{lifeline:six,
  title = {{LIFELINE}: A Quarterly Newsletter for Enthusiasts of John Conway's Game of Life},
  editor = {Wainwright, Robert T.},
  url = {https://conwaylife.com/wiki/Lifeline_Volume_6},
  date = {1972-10},
  number = {Volume 6},
}

@book{adamatzky:book,
  title = {Game of Life Cellular Automata},
  editor = {Adamatzky, Andrew},
  booktitle = {Game of Life Cellular Automata},
  date = {2010},
  doi = {10.1007/978-1-84996-217-9},
  publisher = {Springer London},
}

@software{lifesrc,
  title = {\texttt{lifesrc}},
  author = {Bell, David},
  url = {http://members.tip.net.au/~dbell/programs/lifesrc-3.8.tar.gz},
  date = {2001},
}

@book{winning-ways:vol2,
  title = {{W}inning {W}ays for {Y}our {M}athematical {P}lays},
  author = {Berlekamp, Elwyn R. and Conway, John H. and Guy, Richard K.},
  date = {1982},
  publisher = {Academic Press},
  volume = {2},
}

@software{localforce,
  title = {LocalForce},
  author = {Brown, Nico},
  url = {https://github.com/nicobrownmath/LocalForce},
  date = {2022},
}

@online{brown:strictly-volatile,
  title = {Strictly Volatile Loop},
  author = {Brown, Nico},
  url = {https://conwaylife.com/forums/viewtopic.php?f=2&t=1437&start=3425#p154361},
  date = {2022},
}

@article{byte:facts-of-life,
  title = {Some Facts of Life},
  author = {Buckingham, David J.},
  url = {https://archive.org/details/byte-magazine-1978-12/page/n55/mode/2up},
  date = {1978},
  journaltitle = {Byte Magazine},
  number = {12},
  pages = {54--67},
  volume = {03},
}

@misc{buckingham:conduits,
  title = {My {E}xperience with {B}-heptominos in {O}scillators},
  author = {Buckingham, David J.},
  url = {http://conwaylife.com/ref/lifepage/patterns/bhept/bhept.html},
  date = {1996},
  howpublished = {LifeCA mailing list},
}

@article{buckingham-callahan:bounds,
  title = {Tight {B}ounds on {P}eriodic {C}ell {C}onfigurations in {L}ife},
  author = {Buckingham, David J. and Callahan, Paul B.},
  date = {1998},
  doi = {10.1080/10586458.1998.10504370},
  journaltitle = {Experimental Mathematics},
  number = {3},
  pages = {221--241},
  publisher = {Taylor & Francis, Philadelphia},
  volume = {7},
}

@software{ptbsearch,
  title = {\texttt{ptbsearch}},
  author = {Callahan, Paul},
  url = {https://www.ics.uci.edu/~eppstein/ca/ptbsearch.tar.gz},
  date = {1998},
}

@article{chu-stuckey:maximum-density,
  title = {A complete solution to the {M}aximum {D}ensity {S}till {L}ife {P}roblem},
  author = {Chu, Geoffrey and Stuckey, Peter J.},
  date = {2012},
  doi = {10.1016/j.artint.2012.02.001},
  journaltitle = {Artificial Intelligence},
  pages = {1--16},
  volume = {184},
}

@software{lls,
  title = {Logic Life Search},
  author = {Cunningham, Oscar},
  url = {https://gitlab.com/OscarCunningham/logic-life-search},
  date = {2018},
}

@misc{ekstrom:new-arms,
  title = {New construction arms},
  author = {Ekstr\"{o}m, Simon},
  url = {https://conwaylife.com/forums/viewtopic.php?&p=26024#p26024},
  date = {2015},
}

@online{small-census,
  title = {Census of small evolving patterns},
  author = {Ekstr\"{o}m, Simon},
  url = {https://conwaylife.com/forums/viewtopic.php?&p=38773},
  date = {2016},
}

@misc{flammenkamp:census,
  title = {Top 100 of Game-of-Life Ash Objects},
  author = {Flammenkamp, Achim},
  url = {http://wwwhomes.uni-bielefeld.de/achim/freq_top_life.html},
  date = {2004},
}

@article{gardner:1970,
  title = {{M}athematical {G}ames: {T}he fantastic combinations of {J}ohn {C}onway's new solitaire game ``life''},
  author = {Gardner, Martin},
  date = {1970},
  journaltitle = {Scientific American},
  number = {4},
  pages = {120--123},
  volume = {223},
}

@article{gardner:1971,
  title = {{M}athematical {G}ames: {O}n cellular automata, self-reproduction, the {G}arden of {E}den and the game ``life''},
  author = {Gardner, Martin},
  date = {1971},
  journaltitle = {Scientific American},
  number = {2},
  pages = {112--117},
  volume = {224},
}

@misc{goldtiger:rro,
  title = {Reflectorless Rotating Oscillator Discussion Thread in Life},
  author = {Goldtiger997},
  url = {https://conwaylife.com/forums/viewtopic.php?&p=135135#p135135},
  date = {2021},
}

@misc{gosper:paleoballistics,
  title = {paleoballistics},
  author = {Gosper, Bill},
  date = {2008-11},
  howpublished = {LifeCA mailing list},
}

@misc{lifenews:p31,
  title = {Is Life omniperiodic?},
  author = {Goucher, Adam P.},
  url = {https://web.archive.org/web/20110415012248/http://pentadecathlon.com/lifeNews/2011/01/is_life_omniperiodic.html},
  date = {2011-01-16},
}

@misc{catagolue,
  title = {Catagolue},
  author = {Goucher, Adam P.},
  url = {https://catagolue.hatsya.com/home},
  date = {2015},
}

@misc{apg:sir-robin,
  title = {A rather satisfying winter},
  author = {Goucher, Adam P.},
  url = {https://cp4space.hatsya.com/2018/03/11/a-rather-satisfying-winter/},
  date = {2018},
}

@incollection{apg:universal,
  title = {Universal Computation and Construction in {G}o{L} Cellular Automata},
  author = {Goucher, Adam P.},
  chapter = {25},
  crossref = {adamatzky:book},
  doi = {10.1007/978-1-84996-217-9_25},
  pages = {505--518},
}

@software{dr,
  title = {\texttt{dr}, the drifter search program},
  author = {Hickerson, Dean},
  url = {https://github.com/Matthias-Merzenich/dr/blob/main/dr.documentation.txt},
  date = {1997},
}

@online{hickerson:new-billard-tables,
  title = {New billiard tables},
  author = {Hickerson, Dean},
  url = {https://conwaylife.com/ref/DRH/nbt.html},
  date = {1998},
}

@online{hickerson:stamp-collection,
  title = {Stamp collection},
  author = {Hickerson, Dean},
  url = {https://conwaylife.com/ref/DRH/stamps.html},
  date = {2000},
}

@misc{tollcass:census,
  title = {The Online Life-Like CA Soup Search},
  author = {Johnston, Nathaniel},
  url = {https://web.archive.org/web/20110510032152/https://www.conwaylife.com/soup/census.asp?rule=B3/S23&sl=1&os=1&ss=1},
  date = {2009},
  urldate = {2011-05-11},
}

@book{life:book,
  title = {{C}onway's {G}ame of {L}ife: {M}athematics and {C}onstruction},
  author = {Johnston, Nathaniel and Greene, Dave},
  date = {2022},
  doi = {10.5281/zenodo.6097284},
  pagetotal = {492},
  publisher = {Self-published},
}

@book{taocp:4b,
  title = {Combinatorial Algorithms, Part 2},
  author = {Knuth, Donald},
  date = {2023},
  edition = {1},
  location = {Upper Saddle River, New Jersey},
  maintitle = {The Art of Computer Programming},
  publisher = {Addison-Wesley},
  volume = {4B},
}

@misc{lifenews:p51,
  title = {New Oscillators},
  author = {Koenig, Heinrich},
  url = {https://web.archive.org/web/20090416221124/http://pentadecathlon.com/lifeNews/2009/03/new_oscillators_4.html},
  date = {2009-03-16},
}

@misc{lifenews:p37,
  title = {New Oscillators},
  author = {Koenig, Heinrich},
  url = {https://web.archive.org/web/20090417030009/http://pentadecathlon.com/lifeNews/2009/04/new_oscillators_5.html},
  date = {2009-04-14},
}

@online{catalyst-test,
  title = {Results of catalyst tests},
  author = {Lee, Dongook},
  url = {https://conwaylife.com/forums/viewtopic.php?t=1878},
  date = {2015},
}

@online{leithner:fast,
  title = {p57 and p56 {H}erschel loops},
  author = {Leithner, Dietrich},
  url = {https://conwaylife.com/ref/lifepage/patterns/p57/p57.html},
  date = {1997},
  howpublished = {LifeCA mailing list},
}

@misc{lifewiki:oscillator,
  title = {Oscillator},
  author = {{LifeWiki contributers}},
  url = {https://conwaylife.com/wiki/Oscillator},
  date = {2023},
}

@online{merzenich-snark,
  title = {Re: Just the place for a Snark!},
  author = {Merzenich, Matthias},
  url = {https://conwaylife.com/forums/viewtopic.php?p=7820#p7820},
  date = {2013},
}

@misc{lifewiki:status-page,
  title = {LifeWiki:Game of Life Status page},
  author = {Merzenich, Matthias and others},
  url = {https://conwaylife.com/wiki/LifeWiki:Game_of_Life_Status_page},
  date = {2023},
}

@software{catalyst,
  title = {Catalyst},
  author = {Nivasch, Gabriel},
  url = {https://www.gabrielnivasch.org/fun/life/},
  date = {2001},
}

@software{randomagar,
  title = {RandomAgar},
  author = {Nivasch, Gabriel},
  url = {https://www.gabrielnivasch.org/fun/life/random-agar},
  date = {2002},
}

@online{p23-completion,
  title = {Re: Oscillator Discussion Thread},
  author = {Okanishi, Luka},
  url = {https://conwaylife.com/forums/viewtopic.php?&p=85719#p85719},
  date = {2019},
}

@misc{okrasinski:census,
  title = {LIFESTAT/Life Screen Saver results},
  author = {Okrasinski, Andrzej},
  url = {http://web.archive.org/web/20091027025324/http://geocities.com/conwaylife/},
  date = {2003},
  urldate = {2009-10-27},
}

@software{bellman,
  title = {Bellman},
  author = {Playle, Mike},
  url = {https://github.com/simeksgol/BellmanWin_szlim},
  date = {2013},
}

@online{snark,
  title = {Just the place for a {S}nark!},
  author = {Playle, Mike},
  url = {https://conwaylife.com/forums/viewtopic.php?p=7820},
  date = {2013},
}

@manual{bellman-manual,
  title = {Bellman: a program for finding catalysts in cellular automata},
  author = {Playle, Mike},
  url = {https://github.com/rokicki/lifecontent/raw/master/bellman/bellman-B-2014-08-02.pdf},
  date = {2014},
}

@book{the-recursive-universe,
  title = {The {R}ecursive {U}niverse: {C}osmic {C}omplexity and the {L}imits of {S}cientific {K}nowledge},
  author = {Poundstone, William},
  date = {1985},
  publisher = {Contemporary Books},
}

@software{hdp-spark,
  title = {Re: Golly scripts},
  author = {Raucci, David},
  url = {https://conwaylife.com/forums/viewtopic.php?f=9&t=45&start=275#p130212},
  date = {2021},
}

@book{rendell:turing-machine,
  title = {{T}uring {M}achine {U}niversality of the {G}ame of {L}ife},
  author = {Rendell, Paul},
  date = {2015},
  doi = {10.1007/978-3-319-19842-2},
  publisher = {Springer, Cham},
}

@software{symmetric-catforce,
  title = {Symmetric CatForce},
  author = {Riley, Mitchell},
  url = {https://github.com/mvr/CatForce},
  date = {2022},
}

@online{catalyst-test-mvr,
  title = {Results of catalyst tests},
  author = {Riley, Mitchell},
  url = {https://conwaylife.com/forums/viewtopic.php?p=159968#p158889},
  date = {2023},
}

@inproceedings{rokicki:algorithms,
  title = {Life Algorithms},
  author = {Rokicki, Tomas},
  url = {https://www.gathering4gardner.org/g4g13gift/math/RokickiTomas-GiftExchange-LifeAlgorithms-G4G13.pdf},
  booktitle = {Gathering 4 Gardner 13 Exchange Book},
  date = {2018},
  venue = {G4G13},
}

@inproceedings{salo-torma:oracles,
  title = {{What Can Oracles Teach Us About the Ultimate Fate of Life?}},
  author = {Salo, Ville and T\"{o}rm\"{a}, Ilkka},
  booktitle = {49th International Colloquium on Automata, Languages, and Programming},
  date = {2022},
  doi = {10.4230/LIPIcs.ICALP.2022.131},
  venue = {ICALP 2022},
}

@misc{salo-torma:backwards,
  title = {Computing backwards with Game of Life, part 1: wires and circuits},
  author = {Salo, Ville and T\"{o}rm\"{a}, Ilkka},
  date = {2023},
  eprint = {2308.10198},
  eprintclass = {cs.FL},
  eprinttype = {arXiv},
}

@inreference{lexicon:elbow,
  title = {signal elbow},
  author = {Silver, Stephen},
  url = {https://conwaylife.com/ref/lexicon/lex_s.htm#signalelbow},
  booktitle = {Life Lexicon},
  date = {2018},
}

@software{catforce,
  title = {CatForce},
  author = {Simkin, Michael},
  url = {https://github.com/simsim314/CatForce},
  date = {2014},
}

@software{javalifesearch,
  title = {JavaLifeSearch},
  author = {Suhajda, Karel},
  url = {https://conwaylife.com/forums/viewtopic.php?f=9&t=990},
  date = {2012},
}

@software{summers:torus,
  title = {\texttt{torus}},
  author = {Summers, Jason},
  url = {http://entropymine.com/jason/life/software/},
  date = {2000},
}

@software{winlifesearch,
  title = {WinLifeSearch},
  author = {Summers, Jason},
  url = {https://github.com/jsummers/winlifesearch},
  date = {2012},
}

@software{skopje,
  title = {Skopje},
  author = {Tan, Jeremy},
  url = {https://github.com/Parcly-Taxel/Skopje},
  date = {2021},
}

@inproceedings{wainwright:universal,
  title = {Life is Universal!},
  author = {Wainwright, Robert T.},
  booktitle = {Proceedings of the 7th Conference on Winter Simulation},
  date = {1974},
  doi = {10.1145/800290.811303},
  location = {Washington, DC},
  pages = {449--459},
  publisher = {Winter Simulation Conference},
  series = {WSC '74},
  volume = {2},
}

@misc{wechsler:missions,
  title = {Missions},
  author = {Wechsler, Allan},
  date = {1992-04},
  howpublished = {LifeCA mailing list},
}

@online{reflector-history,
  title = {Re: The search for an Spartan reflector and/or g-x},
  author = {Greene, Dave},
  url = {https://conwaylife.com/forums/viewtopic.php?&p=72760#p72760},
  date = {2019},
}

@online{b3s234-omniperiodic,
  title = {Re: B3/S234},
  author = {Merzenich, Matthias},
  url = {https://conwaylife.com/forums/viewtopic.php?&p=2521#p2521},
  date = {2010},
}

@online{b3s0123-omniperiodic,
  title = {Re: B3/S0123},
  author = {Dean Hickerson},
  url = {https://conwaylife.com/forums/viewtopic.php?&p=139416#p139416},
  date = {1997},
}

@inproceedings{osg07,
  title  = {The open science grid},
  author = {
    Pordes, Ruth
    and Petravick, Don
    and Kramer, Bill
    and Olson, Doug
    and Livny, Miron
    and Roy, Alain
    and Avery, Paul
    and Blackburn, Kent
    and Wenaus, Torre
    and W{\"u}rthwein, Frank
    and Foster, Ian
    and Gardner, Rob
    and Wilde, Mike
    and Blatecky, Alan
    and McGee, John
    and Quick, Rob
  },
  doi       = {10.1088/1742-6596/78/1/012057},
  booktitle = {J. Phys. Conf. Ser.},
  volume    = {78},
  series    = {78},
  pages     = {012057},
  year      = {2007},
}

\end{document}